\documentclass[a4paper,11pt,fleqn]{article}
\usepackage{amsmath,amssymb,amsthm,graphicx,subfigure,float,caption,epstopdf,tabularx,color, bm,amsfonts,epic}
\usepackage[top=1in, bottom=1in, left=1.25in, right=1.25in]{geometry}
\usepackage{appendix}
\usepackage{multirow}
\usepackage{diagbox}
\allowdisplaybreaks
\newtheorem{thm}{Theorem}[section]
\newtheorem{lem}{Lemma}[section]

\newtheorem{rmk}{Remark}[section]
\newtheorem{shm}{Scheme}[section]
\newtheorem*{prf}{Proof}
\numberwithin{equation}{section}

\usepackage{indentfirst}
\graphicspath{{Fig}}

\begin{document}
\title{Explicit high-order energy-preserving methods for general Hamiltonian partial differential equations}
\author{Chaolong Jiang$^1$,\ Yushun Wang$^2$ and Yuezheng Gong$^3$\footnote{Correspondence author. Email:
gongyuezheng@nuaa.edu.cn.}\\
{\small $^1$ School of Statistics and Mathematics, }\\
{\small Yunnan University of Finance and Economics, Kunming 650221, P.R. China}\\
{\small $^2$ Jiangsu Provincial Key Laboratory for NSLSCS,}\\
{\small School of Mathematical Sciences,  Nanjing Normal University,}\\
{\small  Nanjing 210023, P.R. China}\\
{\small $^3$ College of Science,}\\
{\small Nanjing University of Aeronautics and Astronautics, Nanjing 210016, P.R. China}\\
}
\date{}
\maketitle

\begin{abstract}
A novel class of explicit high-order energy-preserving methods are proposed for general Hamiltonian partial differential equations with non-canonical structure matrix. When the energy is not quadratic, it is firstly done that the original system is reformulated into an equivalent form with a modified quadratic energy conservation law by the energy quadratization approach. Then the resulting system that satisfies the quadratic energy conservation law is discretized in time by combining explicit high-order Runge-Kutta methods with orthogonal projection techniques. The proposed schemes are shown to share the order of explicit Runge-Kutta method and thus can reach the desired high-order accuracy. Moreover, the methods are energy-preserving and explicit because the projection step can be solved explicitly. Numerical results are addressed to demonstrate the remarkable superiority of the proposed schemes in comparison with other structure-preserving methods.  \\[2ex]
\textbf{AMS subject classification:} 65M06, 65M70\\[2ex]
\textbf{Keywords:} explicit Runge-Kutta method, orthogonal projection, energy quadratization approach, Hamiltonian system.
\end{abstract}

\section{Introduction}
Hamiltonian partial differential equations (PDEs) play an important role in science and engineering, some particularly important examples include: quantum mechanics, fluid mechanics and electromagnetics ect. { The general form of Hamiltonian PDEs  with independent variables $({\bf x},t)\in\Omega\times [0,T]\subset \mathbb{R}^d\times \mathbb{R},\ d=1,2$ or 3, functions $z$ belonging to a Hilbert space $\mathcal{W}(\Omega)$ with values $z({\bf x},t)=[z_1({\bf x},t),z_2({\bf x},t),\cdots,z_m({\bf x},t)]^T\in\mathbb{R}^m$ is given by
\begin{align}\label{Hamiltionian-system}
\partial_t z = \mathcal{D}(z) \frac{\delta \mathcal{H}(z)}{\delta z},
\end{align}
where $\mathcal{D}(z)$ is a $m$-by-$m$ matrix
operator which is skew-adjoint for all $z$. Here, $\mathcal{H}:=\mathcal{H}(z)$ is the Hamiltonian energy functional, and $\frac{\delta \mathcal{H}(z)}{\delta z}$ is the variational derivative of the Hamiltonian
energy functional with respect to the variable $z$.} One of the most famous geometric characteristics of \eqref{Hamiltionian-system} is that the exact flow has the invariant (also called first integral) $\mathcal{H}$=const.  A numerical scheme that preserves one or more invariants of Hamiltonian PDEs \eqref{Hamiltionian-system} is known as an energy-preserving scheme or integral-preserving scheme. During the past decade, it has been shown that non-energy-preserving schemes may easily show nonlinear blow-up or lead to instability (see Ref. \cite{FM2011}). This is because such schemes may introduce truncation errors that destroy the
physical law numerically. In addition, the energy-preserving property has been showed to be a crucial role in the proof of stability, convergence, existence and uniqueness of the solution for numerical methods (e.g., see Ref. \cite{LQ95}).

As a matter of fact, over the years, there has been an increasing interest in energy-preserving numerical methods for convertive systems. In Ref. \cite{Cooper87}, Cooper proved that all RK methods conserve linear invariants and an irreducible RK method can preserve all quadratic invariants if and only if their coefficients (${\bm A}\in\mathbb{R}^{s\times s},\ {\bm b}\in\mathbb{R}^s$) satisfy $b_ia_{i,j}+b_ja_{j,i}-b_ib_j=0$ for all $i,j=1,\cdots,s.$ However, no RK method can preserve arbitrary polynomial invariants of degree 3 or higher of arbitrary vector fields \cite{CIZ97}. To overcome this difficulty, various different numerical methods which can preserve general invariants are proposed such as the discrete gradient method \cite{MQR99} (including the averaged vector field (AVF) method \cite{CS19jcp,CGM12,GCW14b,LW16b,QM08}), discrete variational derivative methods \cite{DO11,Furihata01,MF01jcp}, the local energy-preserving methods \cite{CS19jcp,JCW19jsc,WWQ08} and the Kahan's method \cite{ELS19}. However, to our best knowledge, most of the existing energy-preserving schemes are only second order in time, which can't provide long time accurate solutions with a given large time step. Thus, how to design high-order and energy-preserving numerical schemes for conservative systems has attracted much attention in recent years. The noticeable ones include the high-order AVF methods \cite{LWQ14,QM08,WWpla12}, Hamiltonian Boundary Value Methods (HBVMs) \cite{BCMR12,BI16,BIT10}, continuous stage Runge-Kutta (CSRK) methods \cite{CH11bit,H10,MB16,TS12} as well as projection methods \cite{BIT12siam,Hairer00bit,Kojima16bit}. More recently, Jiang et al \cite{JWG19} developed a class of arbitrarily high-order energy-preserving schemes for Camassa-Holm equation by combining the methodology of the invariant energy quadratization (IEQ) approach introduced in Refs. \cite{GZYW18,YZW17,ZYGW17} with a symplectic RK method (see e.g., Refs. \cite{Sanz-Sernabit88,SCbook94}). Other high-order energy-preserving schemes can be found in Refs. \cite{CMMOQWesiam09,Matsuo03,ZhangQSaml19}. Despite the exciting high-order schemes are energy-preserving and achieve high-order accuracy in time, all of them are fully implicit which further implies that the practical implementation of the methods is complicated and expensive unless a fast nonlinear solver is proposed \cite{BITjcam11,MB16}. 

Compared to the fully implicit methods, explicit ones are simple and easy to implementation. Thus, in the past few decades, there have been many existing attempts to develop explicit high-order energy-preserving methods. In Ref. \cite{BM02jcam}, del Buono and Mastroserio proposed an explicit fourth-order rational RK method which preserve a quadratic invariant. Later on, Calvo et al. \cite{CHMR06} developed a novel class of high-order explicit methods to preserve quadratic invariants in the numerical integration of the underlying system by using the incremental direction projection technique associated with explicit RK methods. Further studies on such method have been carried out in Ref. \cite{CLMR15}. However, for general invariants, their methods need to solve a nonlinear equations, at every time level, which may lead to expensive costs. Recently, Zhang et al. \cite{ZhangQYS19} constructed a class of explicit high-order numerical methods that can preserve a quadratic invariant for the perturbed Kepler two-body system and the one dimensional nonlinear Schr\"odinger equation, respectively, by using the incremental direction projection and the EQ strategy. In this paper, we will propose a framework for developing explicit high-order energy-preserving methods for general Hamiltonian PDEs \eqref{Hamiltionian-system}, based on the idea of the IEQ approach and the orthogonal projection technique. We first utilize the idea of the IEQ approach to reformulate the system \eqref{Hamiltionian-system} where the energy is not quadratic, into a reformulated system, which inherits a quadratic invariant. Then, the resulting system is solved by the orthogonal projection method associated with explicit RK methods. We show that the Lagrange multiplier of the projected
methods can be explicitly obtained, and the resulting methods can retain the order of the RK method. Thus, our methods are energy-preserving, explicit and can achieve desired high-order accuracy. Moreover, different from  Ref. \cite{ZhangQYS19}, in this paper, we mainly focus on the Hamiltonian PDEs \eqref{Hamiltionian-system} where the invariant is boundness (i.e., defined in the senses of a norm). This is because the boundness of  numerical solution can be directly obtained by these discrete invariants. Therefor it is valuable to expect that the proposed methods for this class of systems will produce richer information. For illustration purposes, we solve the two dimensional nonlinear Schr\"odinger equation and the one and two dimensional sine-Gordon equation to demonstrate the effectiveness of the selected new scheme, respectively. To show their accuracy and efficiency, we also compare our proposed scheme with the Gauss collocation method and the one provided by incremental direction projection.


The rest of this paper is organized as follows. In Section \ref{Sec:PM:2}, we use the idea of the IEQ approach to reformulate the system \eqref{Hamiltionian-system} into an equivalent form. In Section \ref{Sec:PM:3},
the explicit high-order energy-preserving schemes are introduced, and their energies-preservation are discussed. In Sections \ref{Sec:PM:4}, several numerical examples are shown to illustrate the power of our proposed explicit high-order schemes. We draw some conclusions in Section \ref{Sec:PM:7}.

\section{Model reformulation}\label{Sec:PM:2}
In this section, we apply the IEQ approach to reformulate the Hamiltonian PDEs \eqref{Hamiltionian-system} when the energy is not quadratic. The reformulated model satisfies a quadratic energy conservation law, which is equivalent to the original system in the continuous level. This process provides an elegant platform for developing explicit high-order energy-preserving schemes.

For the purpose of illustration, we assume the energy is given by the following
 \begin{align}\label{H-energy:2.1}
 \mathcal{H}=\frac{1}{2}(z,\mathcal{B}z) + \big(f(z),1\big),
 \end{align}
 where $\mathcal{B}$ is a linear, self-adjoint, positive definite operator, and $f(z)$ {is bounded from below (i.e., $f(z)\ge B_0$ for all $z$)} that only depends on $z$ itself, but not its spatial derivatives. {Here, $(\cdot,\cdot)$ represents the inner product defined by $(f,g)=\sum_{i=1}^m\int_{\Omega}f_ig_id{\bf x}$ for any $f,g\in(L^2(\Omega))^m$ and denote $\|\cdot\|$ as the $L^2({\Omega})$-norm, i.e., $\|f\|^2=(f,f)$.}

 Then, the energy \eqref{H-energy:2.1} can be rewritten as a quadratic form
 \begin{align}\label{H-energy:2.2}
 \mathcal{H}=\frac{1}{2}(z,\mathcal{B}z)+\frac{1}{2}\|q\|^2-C_0,
 \end{align}
 by introducing an auxiliary variable $q=\sqrt{2\Big(f(z)+\frac{C_0}{|\Omega|}\Big)},$ where $C_0$ is a constant large enough to make $q$ well-defined for all $z$ and {$|\Omega|=\int_{\Omega}d{\bf x}$}.

Denote $g(z) = \frac{f'(z)}{\sqrt{2\Big(f(z)+\frac{C_0}{|\Omega|}\Big)}}.$ We then reformulate the system \eqref{Hamiltionian-system} into an equivalent form
 \begin{align}\label{IEQ-Formulation}
 \left\lbrace
  \begin{aligned}
&\partial_t z = \mathcal{D}(z)\Big(\mathcal{B}z + g(z) q\Big),\\
&\partial_t q = g(z) \partial_t z,
 \end{aligned}\right.
  \end{align}
  with the consistent initial conditions
   \begin{align}\label{initial-condition}
 z({\bf x},0)=z_0({\bf x}),\ q({\bf x},0)=\sqrt{2\Big(f\big(z({\bf x},0)\big)+\frac{C_0}{|\Omega|}\Big)}.
  \end{align}
Denoting $
\Phi=\left[\begin{array}{cc}
              z \\
              q\\
             \end{array}
\right]$, the system \eqref{IEQ-Formulation} then can be written as the following compact form
\begin{align}\label{IEQ-Formulation-1}
\partial_t \Phi = \mathcal{G}(\Phi)\frac{\delta \mathcal{H}}{\delta \Phi},
\end{align}
with a quadratic energy \eqref{H-energy:2.2} and a modified structure matrix
{\begin{align*}
\mathcal{G}(\Phi) = \left[\begin{array}{cc}
1 \\
g(z)\\
\end{array}
\right] \mathcal{D}(z)
\Big[ 1, ~ g(z)\Big].
\end{align*}}
It is worth mentioning that $\mathcal{G}$ is still skew-adjoint for any $\Phi$, so that the quadratic energy conservation law is satisfied by the reformulated system \eqref{IEQ-Formulation-1} (or \eqref{IEQ-Formulation})
\begin{align}
\frac{d \mathcal{H}}{dt} = \left(\frac{\delta \mathcal{H}}{\delta \Phi} , \partial_t \Phi \right) =  \left(\frac{\delta \mathcal{H}}{\delta \Phi} , \mathcal{G}(\Phi)\frac{\delta \mathcal{H}}{\delta \Phi} \right) = 0.
\end{align}

\begin{rmk}\label{rmk2.1}
On the one hand, the energy quadratization reformulation is not necessary for conservation systems whose invariant is quadratic. For example, the mass of the nonlinear Schr\"{o}dinger equation is a quadratic invariant (see Section \ref{Sec:PM:4}).  On the other hand, the energy quadratization approach can also work for a more general $f$ which depends on $z$ and its spatial derivatives. If $f$ is unbounded from below, we can use the splitting strategy to divide $f$ into several differences which are bounded from below. Then the energy can be transformed into a quadratic form by introducing multiple auxiliary variables and the corresponding model reformulation can be derived (see \cite{JGCW19}). In addition, apart from the IEQ approach, the scalar auxiliary variable (SAV) approach proposed in Refs. \cite{SXY18,SXY19siamrev} is also an efficiently strategy to obtain  the reformulated models which admits a quadratic energy conservation law. {Here, we should note that the modified energy \eqref{H-energy:2.2} and the reformulated system \eqref{IEQ-Formulation-1} are equivalent to the original energy and system in the continuous level, but not for the discrete case. }
\end{rmk}

Since the EQ-reformulated form in \eqref{IEQ-Formulation-1} has a quadratic energy, we next discuss how to devise explicit high-order energy-preserving schemes for it.

 \section{Explicit temporal semi-discrete high-order energy-preserving methods}\label{Sec:PM:3}
 In this section, a class of explicit high-order energy-preserving methods is proposed for the energy-quadratized system \eqref{IEQ-Formulation-1} by utilizing explicit high-order RK methods and orthogonal projection techniques. For simplicity of notations, we denote $\mathcal{F}(\Phi) = \mathcal{G}(\Phi)\frac{\delta \mathcal{H}}{\delta \Phi}$ and then rewrite the EQ reformulation \eqref{IEQ-Formulation-1} into the following form
\begin{align}\label{model-general-field}
\partial_t \Phi = \mathcal{F}(\Phi),
\end{align}
which conserves the quadratic energy $\mathcal{H} = \frac{1}{2}(\Phi,\mathcal{L}\Phi)$ with $\mathcal{L} = \textrm{diag}(\mathcal{B},1).$ {We here mainly focus on developing temporal semi-discrete methods, thus, in the following discussions, the spatial variables of the PDE system \eqref{model-general-field} are still continuous.} Let $M$ is a positive integer and denote $t_{n}=n\tau,$ $n = 0,1,2,\cdots,M$ where $\tau=\frac{T}{M}$ is the time step. The approximation of the function $\Phi({\bf x},t)$ at time $t=t_n$ is denoted by $\Phi^n:=\Phi^n({\bf x}).$

\subsection{Explicit schemes for modular Hamiltonian PDE systems}\label{subsec:scheme1}
In this subsection, we consider the system \eqref{model-general-field} with the special energy $\mathcal{H} = \frac{1}{2}\|\Phi\|^2$, which is called a modular Hamiltonian system. Combining the explicit RK methods and the orthogonal projection technique, we obtain the following explicit energy-preserving methods for the modular conservative system:
\begin{shm}\label{scheme1}
	For given $\Phi^n$, $\Phi^{n+1}$ is calculated by the following two steps
	\begin{enumerate}
		\item Explicit RK: we compute $\widetilde{\Phi}^{n+1}$ using
		\begin{align}\label{ERK-step}
		\left\lbrace
		\begin{aligned}
		&k_1 = \mathcal{F}(\Phi^n),\quad k_i = \mathcal{F}\left(\Phi^n+\tau\sum_{j = 1}^{i-1}a_{ij}k_j\right), ~i = 2,\cdots,s,\\
		&\widetilde{\Phi}^{n+1} = \Phi^n+\tau\sum_{i = 1}^{s}b_{i}k_i,
		\end{aligned}\right.
		\end{align}
		where $b_i, a_{ij}$ are RK coefficients.
		\item Projection: we update $\Phi^{n+1}$ via
		\begin{align}\label{projection1}
		\Phi^{n+1} = \frac{\|\Phi^n\|}{\|\widetilde{\Phi}^{n+1}\|}\widetilde{\Phi}^{n+1},
		\end{align}	
{where $\|\Phi^n\|^2=(\Phi^n,\Phi^n)$ and $\|\widetilde{\Phi}^n\|^2=(\widetilde{\Phi}^n,\widetilde{\Phi}^n)$.}
	\end{enumerate}
\end{shm}

\begin{thm}\label{thm-3.1}
	If the explicit RK step has order $p$ and $\Phi^n \neq 0$, then there exists $\tau^*>0$ such that Scheme \ref{scheme1} has at least order $p$ for all $\tau\in(0,\tau^{*}]$. Moreover, Scheme \ref{scheme1} preserves the semi-discrete energy conservation law
	 \begin{align}\label{DECL}
	 \mathcal{H}^{n+1} = \mathcal{H}^n, \quad \mathcal{H}^{n} = \frac{1}{2}\|\Phi^{n}\|^2, \quad \forall n\geq 0.
	 \end{align}
\end{thm}

\begin{proof}
	 Since the explicit RK step has order $p$, we have the local error
	\begin{align}\label{local-error1}
	\widetilde{\Phi}^{n+1} = \Phi(t_n+\tau)+C_{p+1}(\Phi^n)\tau^{p+1}+\mathcal{O}({\tau^{p+2}}),
	\end{align}
	which leads to
	\begin{align}
	\|\widetilde{\Phi}^{n+1}\|^2 = \|\Phi(t_n+\tau)\|^2+\mathcal{O}({\tau^{p+1}}).
	\end{align}
{Here, $C_{p+1}(\Phi^n)$ represents a constant which is dependent on the exact solution $\Phi$ at $t=t_n$ and independent of $\tau$.} Then noticing that the modular conservative system \eqref{model-general-field} satisfies the energy conservation $\frac{1}{2}\|\Phi(t_n+\tau)\|^2 = \frac{1}{2}\|\Phi^0\|^2 = \frac{1}{2}\|\Phi^n\|^2,$ thus we obtain
	\begin{align}\label{Phititlde-relation}
	\|\widetilde{\Phi}^{n+1}\|^2 = \|\Phi^n\|^2+\mathcal{O}({\tau^{p+1}}).
	\end{align}
	According to \eqref{Phititlde-relation} and $\Phi^n \neq 0,$ there exists $\tau^*>0$ such that $\|\widetilde{\Phi}^{n+1}\|>0$ for all $\tau\in(0,\tau^{*}]$. Therefore, we can deduce from \eqref{Phititlde-relation}
	\begin{align}\label{projection-para-es}
	\frac{\|\Phi^n\|}{\|\widetilde{\Phi}^{n+1}\|} =1+\mathcal{O}({\tau^{p+1}}).
	\end{align}
	Using \eqref{local-error1}  and \eqref{projection-para-es} leads to
	\begin{align}
	\Phi^{n+1} = \frac{\|\Phi^n\|}{\|\widetilde{\Phi}^{n+1}\|}\widetilde{\Phi}^{n+1} = \Phi(t_n+\tau)+\mathcal{O}({\tau^{p+1}}).
	\end{align}
	The discrete energy conservation \eqref{DECL} is clearly held by following the projection step \eqref{projection1}.
	This completes the proof.
\end{proof}

\begin{rmk}
	If the initial condition is taken as $\Phi^0 = 0,$ then the modular conservation system here satisfies the zero exact solution. So we're only thinking about the case where $\Phi^0$ is non-zero. It's not hard to see that Scheme \ref{scheme1} can work when $\widetilde{\Phi}^{n+1} \neq 0.$ Eq. \eqref{Phititlde-relation} implies $\widetilde{\Phi}^{n+1} \neq 0$ for sufficiently small $\tau$. Therefore, if $\widetilde{\Phi}^{n+1}$ is calculated to be zero in numerical calculations, the time step $\tau$ needs to be reduced to recalculate $\widetilde{\Phi}^{n+1}$ that is not zero.
\end{rmk}

\begin{rmk}
	Quadratic invariants of the form $\mathcal{H}=\frac{1}{2}\|\Phi\|^2$ appear in many physical problems, such as the nonlinear Schr\"{o}dinger equation and the Korteweg-de Vries equation. The proposed methods here will be considered for the nonlinear Schr\"{o}dinger equation in this paper.
\end{rmk}

\begin{rmk}
	As a matter of fact, the projection step \eqref{projection1} is explicit and can be derived by the standard projection method
	\begin{align}
	\frac{1}{2}\|\Phi^{n+1}-\widetilde{\Phi}^{n+1}\|^2\longrightarrow \min \qquad \textrm{subject to} \qquad \frac{1}{2}\|\Phi^{n+1}\|^2 = \frac{1}{2}\|\Phi^n\|^2.
	\end{align}
	However, for general quadratic energy $\mathcal{H} = \frac{1}{2}(\Phi,\mathcal{L}\Phi)$ where $\mathcal{L} \neq I,$ the standard projection method always derives an implicit format, which needs to be solved by a nonlinear iteration.	
\end{rmk}
%

\subsection{Explicit schemes for general Hamiltonian PDE systems}\label{subsec:scheme2}
In this subsection, we are committed to developing explicit energy-preserving methods for the EQ system \eqref{model-general-field} with general quadratic energy $\mathcal{H} = \frac{1}{2}(\Phi,\mathcal{L}\Phi)$. Changing the standard projection method in Scheme \ref{scheme1} to the modified projection technique (see P. 111 of Ref. \cite{ELW06}), we derive the following explicit schemes for general Hamiltonian PDEs:
\begin{shm}\label{scheme2}
	For given $\Phi^n$, $\Phi^{n+1}$ is calculated by the following two steps
	\begin{enumerate}
		\item Explicit RK: we calculate $\widetilde{\Phi}^{n+1}$ through an explicit RK method \eqref{ERK-step}.
		\item Modified projection: we update $\Phi^{n+1}$ via
		\begin{align}\label{projection2}
		\Phi^{n+1} = (I+\lambda_n\mathcal{L})\widetilde{\Phi}^{n+1},
		\end{align}	
		where $\lambda_n$ is a constant given by
		 \begin{align}\label{lambda-value}
		\lambda_n = \frac{-\delta_n}{\beta_n+\sqrt{\beta_n^2-\alpha_n\delta_n}},
		\end{align}
		and $\alpha_n = (\widetilde{\Phi}^{n+1},\mathcal{L}^3\widetilde{\Phi}^{n+1}), ~\beta_n = (\widetilde{\Phi}^{n+1},\mathcal{L}^2\widetilde{\Phi}^{n+1}), ~\delta_n = (\widetilde{\Phi}^{n+1},\mathcal{L}\widetilde{\Phi}^{n+1})-(\Phi^{n},\mathcal{L}\Phi^{n}).$
	\end{enumerate}
\end{shm}

\begin{thm}\label{thm-DECL2}
	If the explicit RK step has order $p$ and $\mathcal{L}\Phi^n \neq 0$, then there exists $\tau^{*}>0$ such that Scheme \ref{scheme2} has at least order $p$ for all $\tau\in(0,\tau^{*}]$. Scheme \ref{scheme2} preserves the discrete energy conservation law
	\begin{align}\label{DECL2}
	\mathcal{H}^{n+1} = \mathcal{H}^n, \quad \mathcal{H}^{n} = \frac{1}{2}(\Phi^n,\mathcal{L}\Phi^n), \quad \forall n\geq 0.
	\end{align}
\end{thm}

\begin{proof}
	  Since the explicit RK step has order $p$, we have the local error
	  \begin{align}\label{local-error2}
	  \widetilde{\Phi}^{n+1} = \Phi(t_n+\tau)+\mathcal{O}({\tau^{p+1}}) = \Phi^n + \mathcal{O}(\tau),
	  \end{align}
	  which leads to
	  \begin{align}
	  & \alpha_n = (\Phi^n,\mathcal{L}^3\Phi^n) + \mathcal{O}(\tau),\label{alpha-err}\\
	  & \beta_n  = (\Phi^n,\mathcal{L}^2\Phi^n) + \mathcal{O}(\tau),\label{beta-err}\\
	  & \delta_n = \big(\Phi(t_n+\tau),\mathcal{L}\Phi(t_n+\tau)\big)-(\Phi^{n},\mathcal{L}\Phi^{n}) + \mathcal{O}({\tau^{p+1}}).
	  \end{align}
	  Noticing the energy conservation $\frac{1}{2}\big(\Phi(t_n+\tau),\mathcal{L}\Phi(t_n+\tau)\big) = \frac{1}{2} (\Phi^{0},\mathcal{L}\Phi^{0}) = \frac{1}{2} (\Phi^{n},\mathcal{L}\Phi^{n}),$ we obtain
	  \begin{align}
	  \delta_n = \mathcal{O}({\tau^{p+1}}). \label{delta-err}
	  \end{align}
	  Using \eqref{alpha-err}, \eqref{beta-err} and \eqref{delta-err}, we can derive
	  \begin{align}
	  \beta_n^2-\alpha_n\delta_n = (\Phi^n,\mathcal{L}^2\Phi^n)^2 + \mathcal{O}(\tau).\label{delta}
	  \end{align}
	{Since $\mathcal{L}$ is self-adjoint and we suppose $\mathcal{L}\Phi^n \neq 0,$ we have $(\Phi^n,\mathcal{L}^2\Phi^n)=(\mathcal{L}\Phi^n,\mathcal{L}\Phi^n)>0,$ for all $n$.} Then, according to \eqref{beta-err} and \eqref{delta}, there exists $\tau^*>0$ such that $\beta_n>0, ~ \beta_n^2-\alpha_n\delta_n >0$ for all $\tau\in(0,\tau^{*}]$, which makes $\lambda_n$ well-defined. In addition, we can derive from Eqs. \eqref{beta-err}, \eqref{delta-err} and \eqref{delta}
	  \begin{align}
	  \lambda_n = \mathcal{O}({\tau^{p+1}}).
	  \end{align}
	  Therefore, Scheme \ref{scheme2} has at least order $p$.
	
	  By a direct calculation, the discrete energy conservation \eqref{DECL2} is readily proved by following the projection step \eqref{projection2}. This completes the proof.
\end{proof}

\begin{rmk}
	For general conservation system, the energy $\mathcal{H} = \frac{1}{2}(\Phi,\mathcal{L}\Phi)$ is generally nonzero, which implies $\mathcal{L}\Phi \neq 0.$ If $\beta_n$ is calculated to be zero or $\beta_n^2-\alpha_n\delta_n \leq 0,$ the time step $\tau$ needs to be reduced to recalculate $\widetilde{\Phi}^{n+1}$ so that $\beta_n>0,  \beta_n^2-\alpha_n\delta_n >0.$
\end{rmk}

\begin{rmk}\label{rmk:3.5}
	In this paper, we focus on the case that the operator $\mathcal{L}$ is self-adjoint and positive definite so that the energy $\mathcal{H}$ is bounded. In this case, the developed energy-preserving methods may produce some unexpected nonlinear stability. On the one hand, we give numerical examples in a classical explicit RK method of order 4, where their coefficients are given by the following Butcher table
\begin{table}[H]
\centering
{\begin{tabular}{c|ccccc}
${c_1}$ & $\ $& $\ $&$\ $&$\ $ \\
${c_2}$ &$a_{21}$& $\ $&$\ $&$\ $ \\
${c_3}$ & $a_{31}$& $a_{32}$&$\ $&$\ $ \\
${c_4}$ & $a_{41}$& $a_{42}$&$a_{43}$&$\ $ \\
\hline
& $b_1$ &$b_2$&$b_3$&$b_4$\\
\end{tabular}
=
\begin{tabular}{c|cccc}
0& \ &\  &\ &\   \\
1/2 & 1/2 &\ &\ &\ \\
1/2 & 0 & 1/2 & \ &\  \\
1& 0 & 0 & 1&\   \\
\hline
& 1/6 &2/6 & 2/6&1/6\\
\end{tabular}.}
\end{table}
\noindent{On the other hand, we have replaced the classical RK method with the well known DOPRI5(4) of Dormand-Prince \cite{DP80jcam} that
allows a variable step size for the proposed schemes in following practice computations.}
\end{rmk}
\begin{rmk} {To prevent the numerical accumulation of roundoff errors in energy, we replace $\Phi^n$ in \eqref{projection1} and \eqref{lambda-value} with $\Phi^0$, respectively, in following numerical simulations.}

\end{rmk}

\section{Numerical results}\label{Sec:PM:4}
In the previous sections, we present some explicit high-order energy-preserving schemes for general Hamiltonian PDEs. In this section, we apply the proposed \textbf{Schemes \ref{scheme1}} and \textbf{\ref{scheme2}} to solve two benchmark Hamiltonian PDE systems, the nonlinear Schr\"odinger (NLS) equation and the sine-Gordon (SG) equation, respectively, where the periodic boundary condition is considered. {For the spatial discretization, we shall pay special attention the following three aspects: } {\begin{itemize}
\item preserve the skew-adjoint property of the operator $\mathcal{G}(\Phi)$ and the self-adjoint, positive definite property of the operator $\mathcal{L}$;
\item preserve the discrete integration-by-parts formulae \cite{DO11};
\item is high-order accuracy which is comparable to that of the time-discrete methods.
\end{itemize}
Based on these statements and the boundary condition, the standard Fourier pseudo-spectral method is employed which is omitted here due to space limitation. Interested readers are referred to Refs.  \cite{CQ01,ST06} for details.}

\subsection{Nonlinear Schr\"odinger equation}\label{sec:4.1}

We consider the nonlinear Schr\"odinger equation given as follows
\begin{align}\label{NLS-equation}
&\text{i}\partial_tu+\Delta u+\beta|u|^2u=0,
  \end{align}
where $\text{i}=\sqrt{-1}$ is the complex unit, $u$ is the complex-valued wave function, $\Delta$ is the usual Laplace operator, and $\beta$ is a given real constant.

If we suppose $u=p+\text{i}q$, the nonlinear Schr\"odinger equation \eqref{NLS-equation} is rewritten into
\begin{align}\label{NLS-equivalent-equation}
\left \{
 \aligned
&\partial_t p=-\Delta q-\beta(p^2+q^2)q,\\
&\partial_t q=\Delta p+\beta(p^2+q^2)p.
\endaligned
 \right.
\end{align}
Denoting $
\Phi=\left[\begin{array}{cc}
              p \\
              q\\
             \end{array}
\right]$, then the system \eqref{NLS-equivalent-equation} is equivalently transformed into the system \eqref{IEQ-Formulation-1}
with a quadratic invariant
\begin{align}\label{NLS-mass-law}
\mathcal{H} = \frac{1}{2}\|\Phi\|^2 =  \frac{1}{2}(\|p\|^2+\|q\|^2),
\end{align}
and a structure matrix
\begin{align*}
 \mathcal{G}(\Phi)=\left[\begin{array}{cc}
              \ & -\Delta-\beta(p^2+q^2) \\
              \Delta+\beta(p^2+q^2)& \ \\
             \end{array}
\right].
\end{align*}

Applying \textbf{Scheme \ref{scheme1}} to system \eqref{NLS-equivalent-equation}, we obtain
\begin{shm}\label{scheme-NLS-equation} For given $(p^n,q^n)$, $(p^{n+1}, q^{n+1})$ is calculated by the following two steps
\begin{enumerate}
		\item {Explicit RK: we compute $(\widetilde{p}^{n+1},\widetilde{q}^{n+1})$ by using the explicit RK4 method (see remark \ref{rmk:3.5}) to the system \eqref{NLS-equivalent-equation}.}
\item Projection: we update $(p^{n+1},q^{n+1})$ via
\begin{align}\label{NLS-explict-projection}
p^{n+1} = \sqrt{\frac{\|p^n\|^2+\|q^n\|^2}{\|\widetilde{p}^{n+1}\|^2+\|\widetilde{q}^{n+1}\|^2}} \widetilde{p}^{n+1},\ q^{n+1} = \sqrt{\frac{\|p^n\|^2+\|q^n\|^2}{\|\widetilde{p}^{n+1}\|^2+\|\widetilde{q}^{n+1}\|^2}} \widetilde{q}^{n+1}.
\end{align}	
\end{enumerate}
\end{shm}

For simplicity of notation, we shall introduce the discrete quadratic invariant in
one space dimension, i.e. $d=1$ in \eqref{NLS-equivalent-equation} and Section \ref{sec:4.2}. Generalizations to $d=2$ are straightforward. Choose the spatial domain $\Omega=[a,b]$ and the mesh size $h=(b-a)/N$, where $N$ is an even positive integer, and denote the grid points by $x_{j}=jh$ for $j=0,1,2,\cdots,N$; let $P_{j}^n$ and $Q_{j}^n$ be the numerical approximations of  $p(x_j,t_n)$ and $q(x_j,t_n)$ for $n=0,1,2,\cdots,M$ and $j=0,1,2,\cdots,N$, and $P^n:=[p_0^n,p_1^n,\cdots,p_{N-1}^n]^T$, $Q^n:=[q_0^n,q_1^n,\cdots,q_{N-1}^n]^T$ be the numerical solution vectors. In addition, for any grid function vectors $U^n$ and $V^n$, we define the discrete inner product and $L^2$-norm as follows
\begin{align*}
&\langle U^n,V^n\rangle_{h}=h\sum_{j=0}^{N-1}U^n_{j}V^n_{j},\ ||U^n||_{h}^2=h\sum_{j=0}^{N-1}(U^n_{j})^2.
\end{align*} According to Theorem \ref{thm-3.1}, the {\bf Scheme 3.1} preserves the following discrete mass conservation law
{\begin{align}\label{IFD-energy-conservation-law}
H^{n}=\frac{1}{2}\big(\|P^n\|_h^2+\|Q^n\|_h^2\big),\ \ n=1,\cdots,M.
\end{align}
In order to quantify the residuals of the discrete mass conservation law, we use the relative mass residual between the discrete mass (i.e., ${H}^n$) at $t=t_n$ and the initial discrete ones, respectively, as
\begin{align}
RM^n=\Big|\frac{{H}^n-{H}^0}{{H}^0}\Big|,\ n=0,1,2,\cdots,M.
\end{align}}

\begin{rmk}\label{rmk:4.1} {Since the original energy is quadratic, the {\bf Scheme 4.1} can preserve the discrete energy in the original variables.}
\end{rmk}

First of all, we present the time mesh refinement tests to show the order of accuracy of the proposed scheme. Also, the following high-order schemes for preserving the discrete quadratic energy are chosen for comparisons:
\begin{itemize}
\item 4th-order GM: the 2-stage Gauss method described in Refs. \cite{Sanz-Sernabit88,SCbook94};
\item 4th-order PM: the projection method for quadratic invariants proposed in Refs. \cite{CHMR06,ZhangQYS19}.
\item {DOPRI5(4): we replace the explicit RK4  method (see Remark \ref{rmk:3.5}) with the well known DOPRI5(4) of Dormand-Prince in Ref. \cite{DP80jcam} at each time level. Note that, in our computations, we implement DOPRI5(4) by the matlab library function ode45.m.}
\end{itemize}

As a summary, the properties of these schemes have been given in Tab. \ref{Tab-IPs:1}.
\begin{table}[H]
\tabcolsep=6pt
\small
\renewcommand\arraystretch{1.12}
\centering
\caption{Comparison of properties of different numerical schemes}\label{Tab-IPs:1}
\begin{tabular*}{\textwidth}[h]{@{\extracolsep{\fill}}c c c c c c c c c }\hline 
 \diagbox{Property}{Scheme}& 4th-order GM& 4th-order PM& The proposed method \\\hline
  Conserving quadratic energy& Yes& Yes &Yes \\[1ex]
 Temporal accuracy& 4th&4th&4th \\[1ex]
 Explicit &No&Yes&Yes\\
\hline
\end{tabular*}
\end{table}

The first test example is the nonlinear Schr\"{o}dinger equation \eqref{NLS-equation} in one dimension, which has a single soliton solution given by \cite{HV2003}
\begin{align}
u(x,t)=\sqrt{\frac{2\alpha}{\beta}}e^{\text{i}(\frac{1}{2}cx-(\frac{1}{4}c^2-\alpha)t)}\text{sech}(\sqrt{\alpha}(x-ct)),\ x\in\mathbb{R},\ t>0,
\end{align}
where $\alpha$ and $c$ are two real parameters.

The computations are done on the space interval $\Omega=[-40,40]$ and choose parameters $\alpha=1,\ \beta=2$ and $c=4$. To test the temporal discretization errors of the three numerical schemes, we fix the Fourier node $800$ so that spatial errors play no role here.
\begin{figure}[H]
\centering\begin{minipage}[t]{60mm}
\includegraphics[width=60mm]{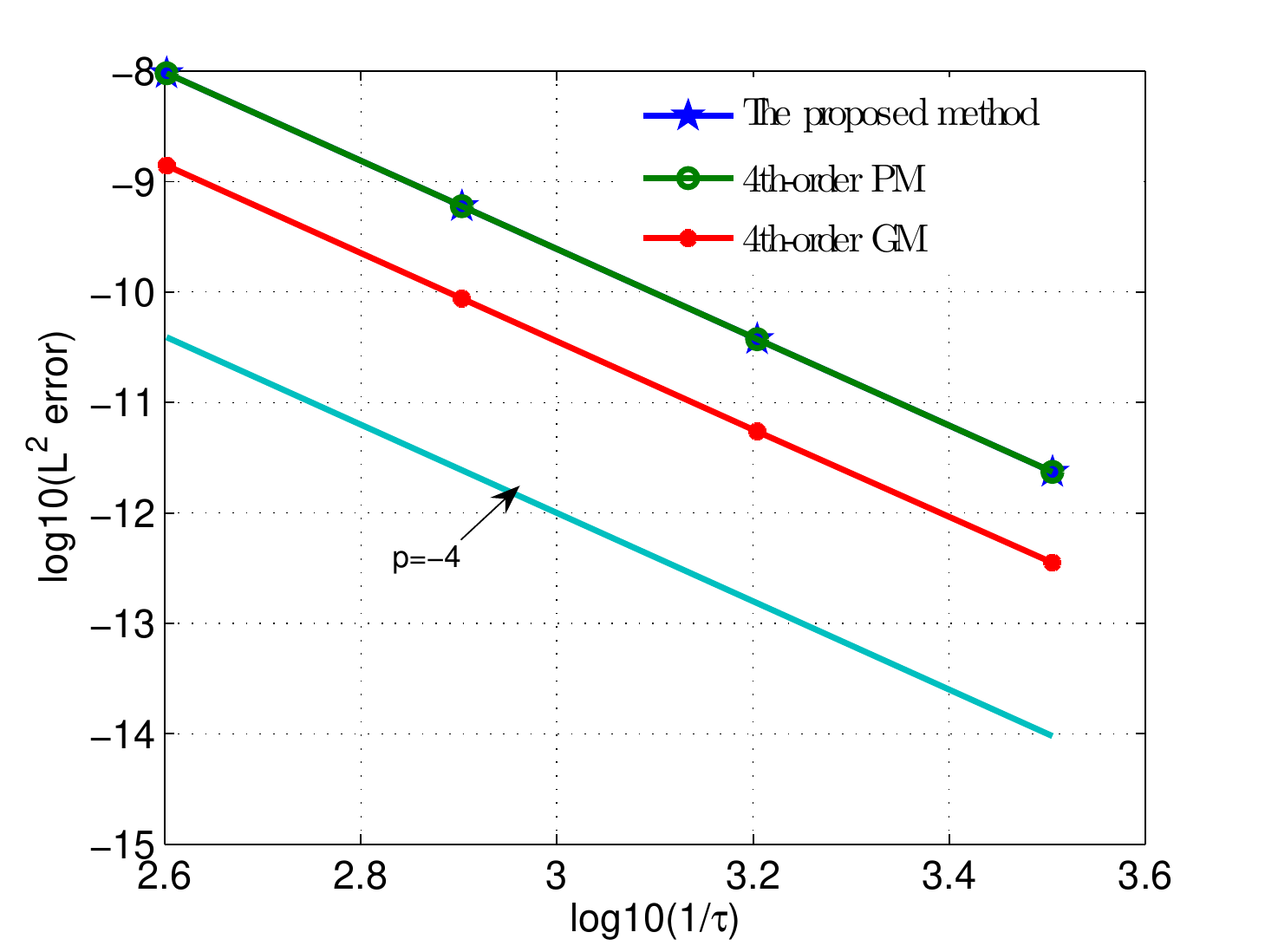}
\end{minipage}
\begin{minipage}[t]{60mm}
\includegraphics[width=60mm]{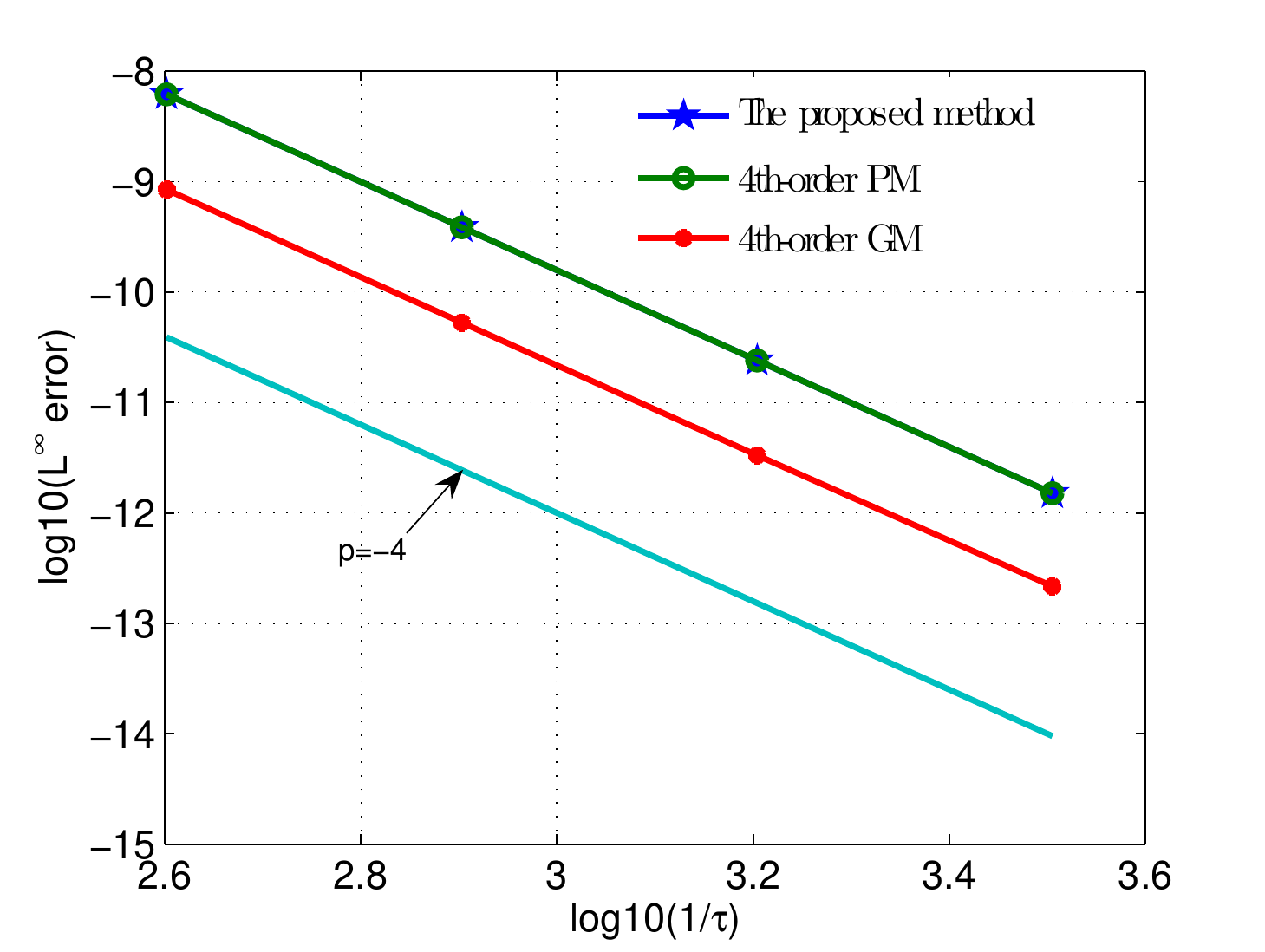}
\end{minipage}
\caption{ Time step refinement tests using the different numerical schemes for the one dimensional Schr\"odinger equation \eqref{NLS-equation}.}\label{1d-scheme:fig:1}
\end{figure}

\begin{figure}[H]
\centering\begin{minipage}[t]{60mm}
\includegraphics[width=60mm]{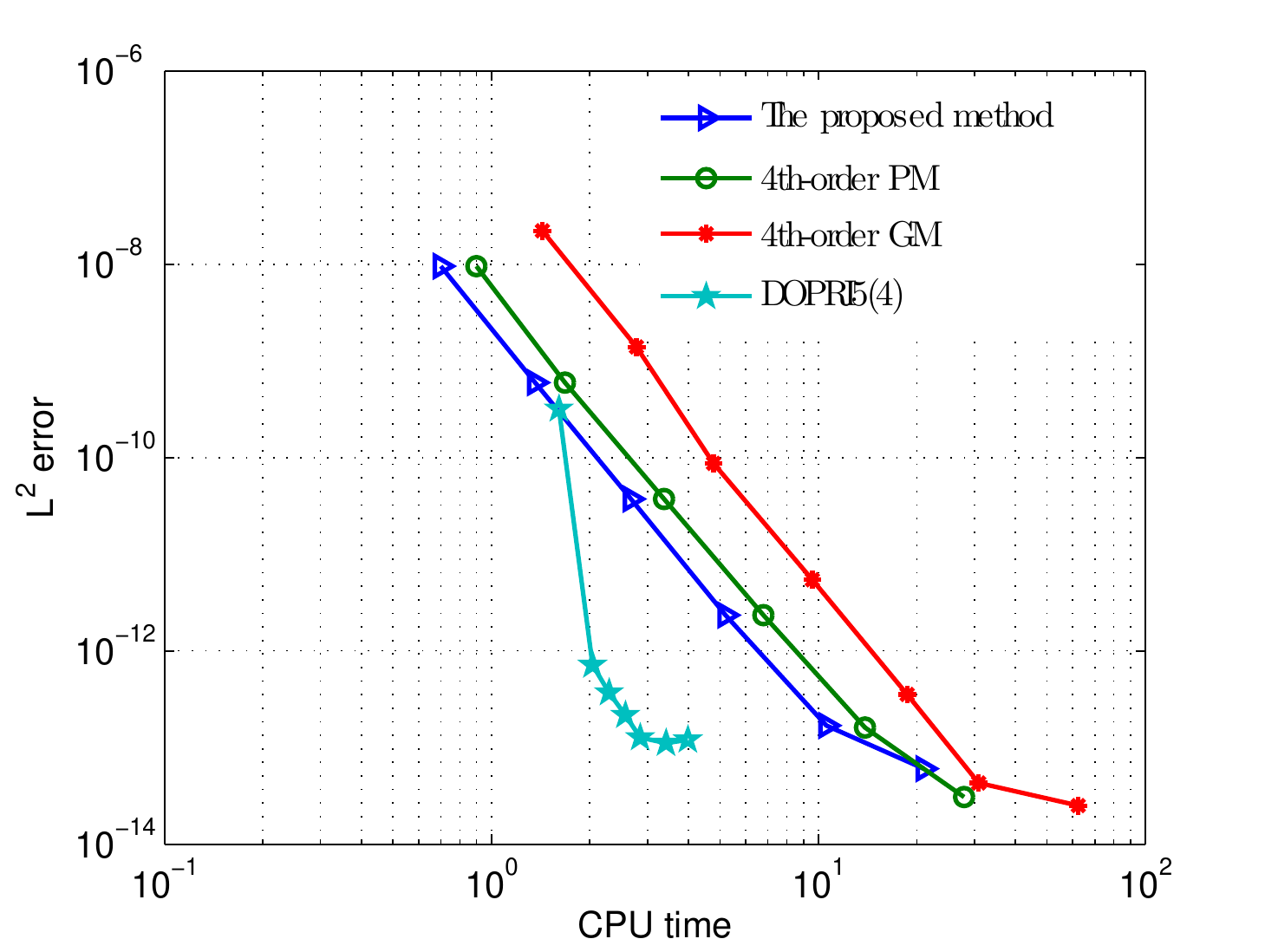}
\end{minipage}
\begin{minipage}[t]{60mm}
\includegraphics[width=60mm]{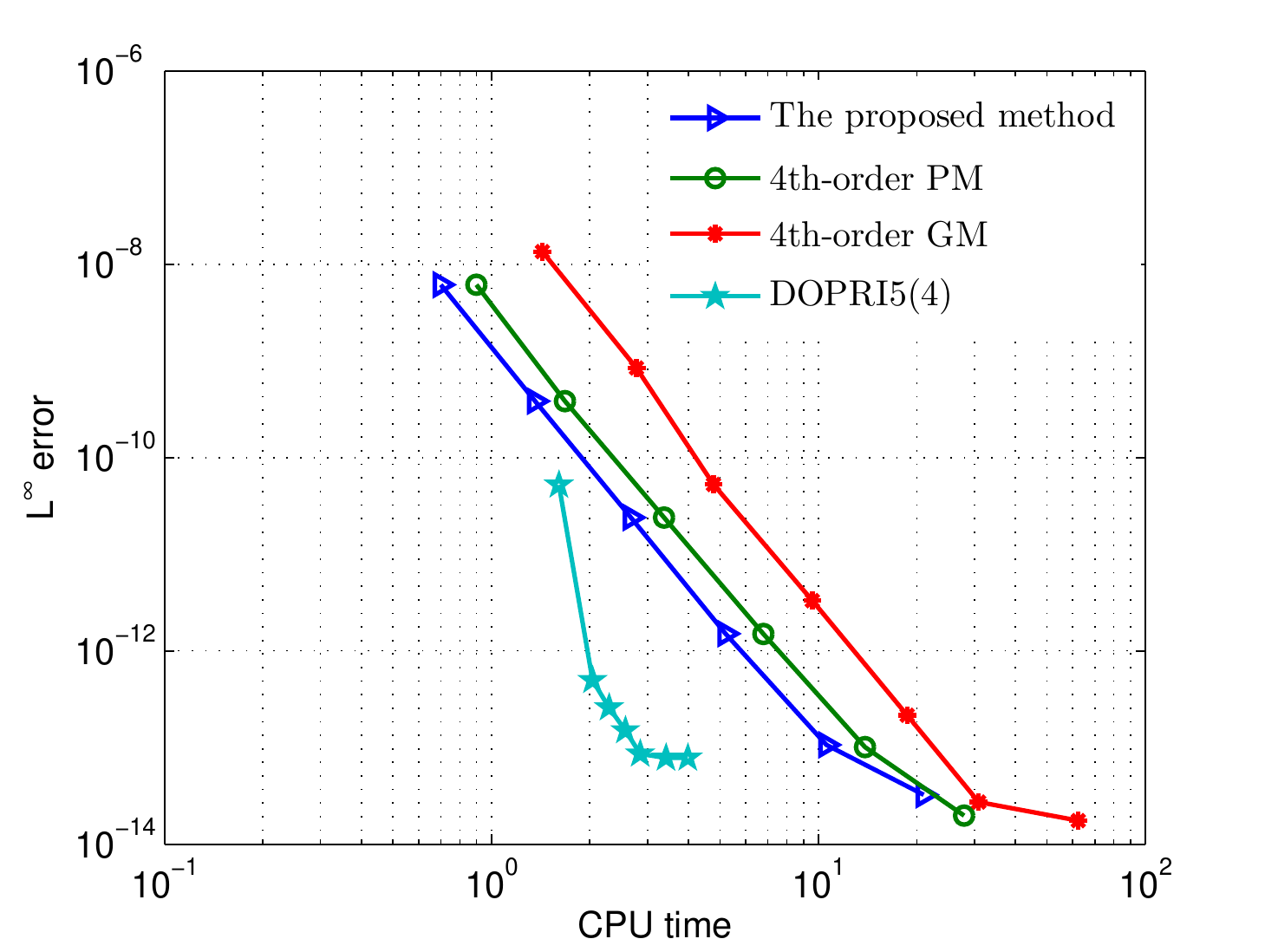}
\end{minipage}
\caption{ The numerical error versus the CPU time using the different numerical schemes for the one dimensional Schr\"odinger equation \eqref{NLS-equation}.}\label{1d-scheme:fig:2}
\end{figure}

\begin{figure}[H]
\centering
\begin{minipage}[t]{70mm}
\includegraphics[width=70mm]{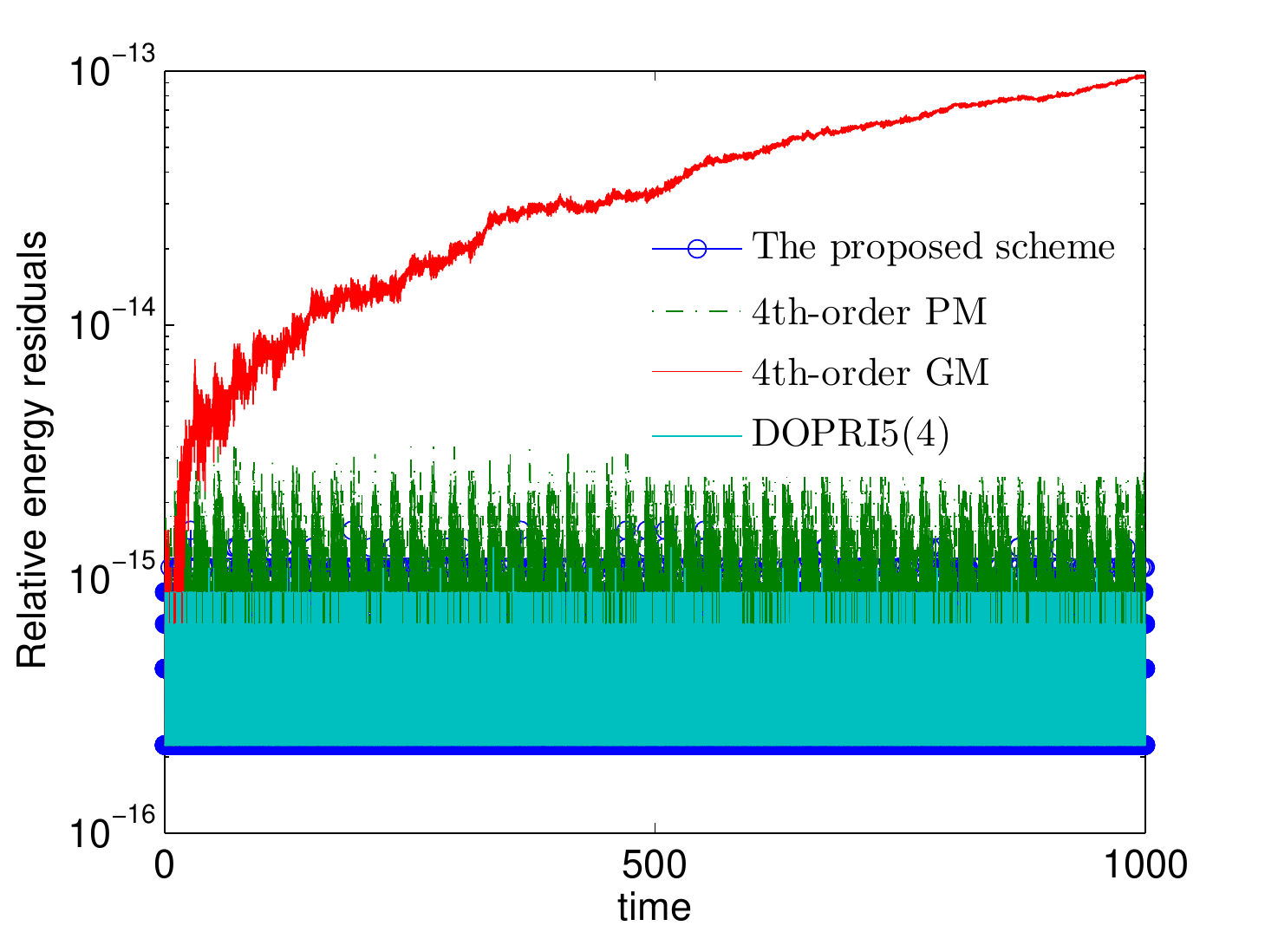}
\end{minipage}
\caption{The relative energy residuals using the different numerical schemes with time step $\tau=0.01$ and Fourier node 256 for the one dimensional Schr\"odinger equation \eqref{NLS-equation}.}\label{1d-scheme:fig:3}
\end{figure}

The $L^2$ errors and $L^{\infty}$ errors in numerical solution of $u$ in short time (i.e., $t=1$) are calculated using three different numerical schemes with various
time steps {$\tau= 0.0025, 0.00125, 0.000625$ and $0.0003125$,} and the results are displayed in Fig. \ref{1d-scheme:fig:1}. In Fig. \ref{1d-scheme:fig:2}, we show the global $L^2$ errors and $L^{\infty}$ errors of $u$ versus the CPU time using the different schemes at $t=1$ with the Fourier node $800$. From Figs. \ref{1d-scheme:fig:1} and \ref{1d-scheme:fig:2}, we can draw the following observations: (i) all methods have fourth order accuracy in time; (ii) the error provided by the 4th-order GM is smallest, and the one provided by the proposed scheme has the same order of magnitude as the one of the 4th-order PM; (iii) for a given global error, the cost of 4th-order GM is most expensive while the one of DOPRI5(4) is cheapest. {Although the proposed scheme is much cheaper than the 4th-order PM , the efficiency improvement of the proposed scheme
is not even 1 order of magnitude better than the 4th-order PM. The following numerical experiments also
appear similar results. This is due to that the 4th-order PM is also fully explicit scheme.}

To further investigate the energy-preservation
of the proposed scheme, we provide the energy errors using the different numerical schemes for the 1D Schr\"odinger equation over the time interval $t\in[0,1000]$ in Fig. \ref{1d-scheme:fig:3}, which shows that all methods can preserve the discrete quadratic invariant \eqref{IFD-energy-conservation-law} up to round-off error and the 4th-order GM admits largest energy error. {According to Theorem \ref{thm-3.1} and Ref. \cite{CHMR06}, the residual in the energy for the explicit schemes under consideration should be zero, but it remains almost constant along the time in Fig. \ref{1d-scheme:fig:3}. The following numerical experiments also appear similar situations. We thank that it is due to the round off error introduced by the numerical simulations.} In addition, the residual in energy of the 4th-order GM looks like to show linear decrease in long time computation, we suspect that it is due to solve the discrete nonlinear system by using the iterative method.

{We then test the superposition of two solitons, a slower one ahead of a faster one, such that initially they are well separated, for the 1D nonlinear Schr\"odinger equation \eqref{NLS-equation}.
The initial profile is given by \cite{HV2003}
\begin{align}
u(x,0)=\sqrt{\frac{2\alpha}{\beta}}e^{\text{i}\frac{1}{2}c_1x}\text{sech}(\sqrt{\alpha}x)+e^{\text{i}\frac{1}{2}c_2(x-\delta)}\text{sech}(\sqrt{\alpha}(x-\delta))
\end{align}
with $\alpha=\frac{1}{2},\ \beta=1,\ c_1=1,\ c_2=\frac{1}{10}$ and $\delta=25$.}  We take the space interval $\Omega=[-20,80]$ and the the superposition of two solitons are showed in Fig. \ref{1d-scheme:fig-3}.  {As illustrated in Fig. (a), the collision between two solitons seems quite elastic, and the solitary waves propagate in their original directions with the same velocities. which can be verified in Fig. (b) again. The simulated solutions are precisely consistent to existing results in Ref. \cite{HV2003}.} In Fig. \ref{1d-scheme:fig:4}, we present the residuals in the quadratic energy \eqref{IFD-energy-conservation-law} over the time interval $t\in[0,44]$ calculated
with different numerical schemes and the results, which shows that all of the schemes can preserve the quadratic energy exactly. {It is remarked that, when fixing the spatial step and enlarged the  time step (i.e., $\tau=0.004$), we find that the numerical results provided by the explicit schemes are wrong, this is because for some explicit methods some bounds on the step size could be necessary
to control high order harmonics.}

\begin{figure}[H]
\centering\begin{minipage}[t]{60mm}
\includegraphics[width=60mm]{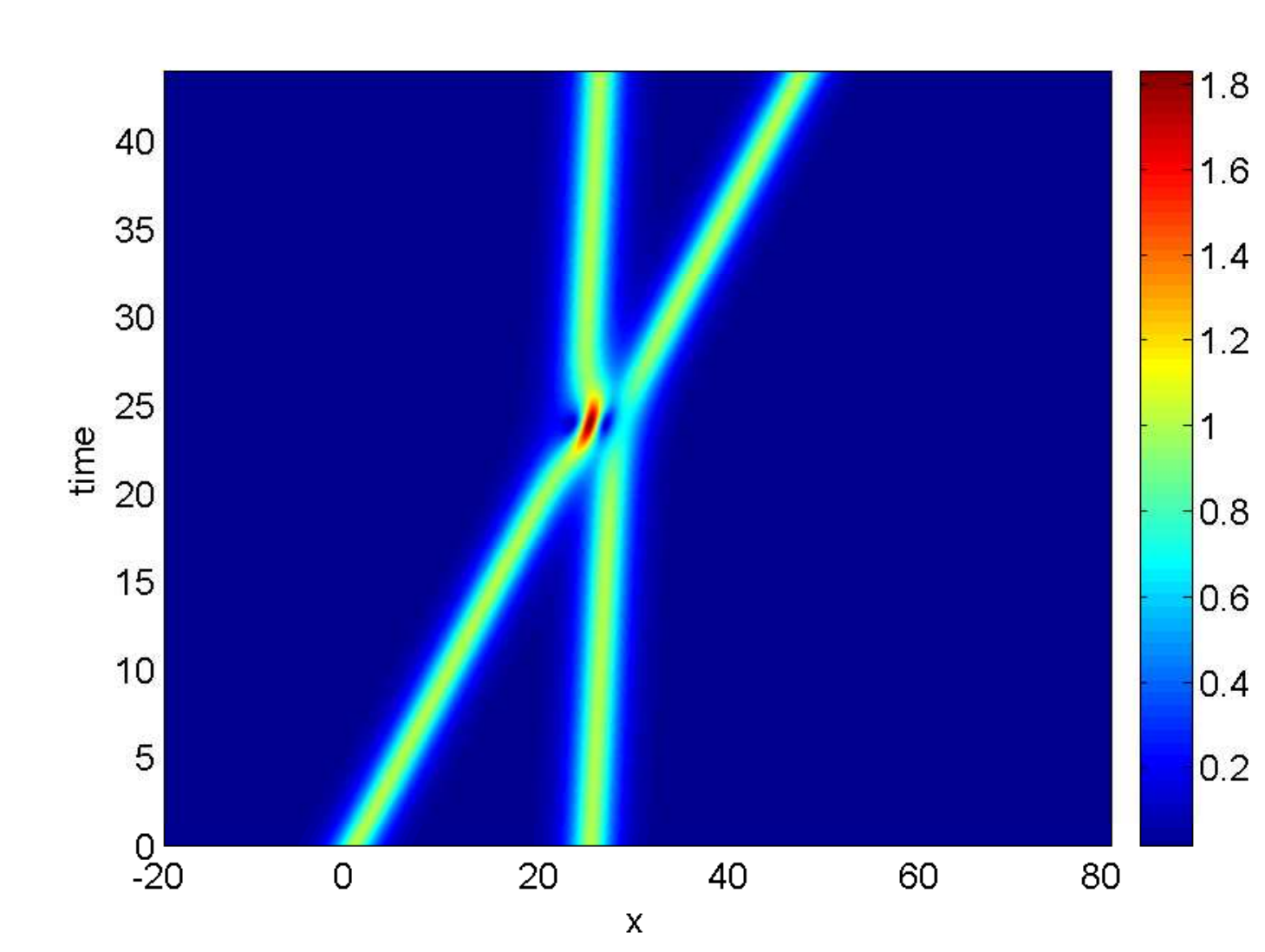}
\caption*{(a)}
\end{minipage}
\begin{minipage}[t]{60mm}
\includegraphics[width=60mm]{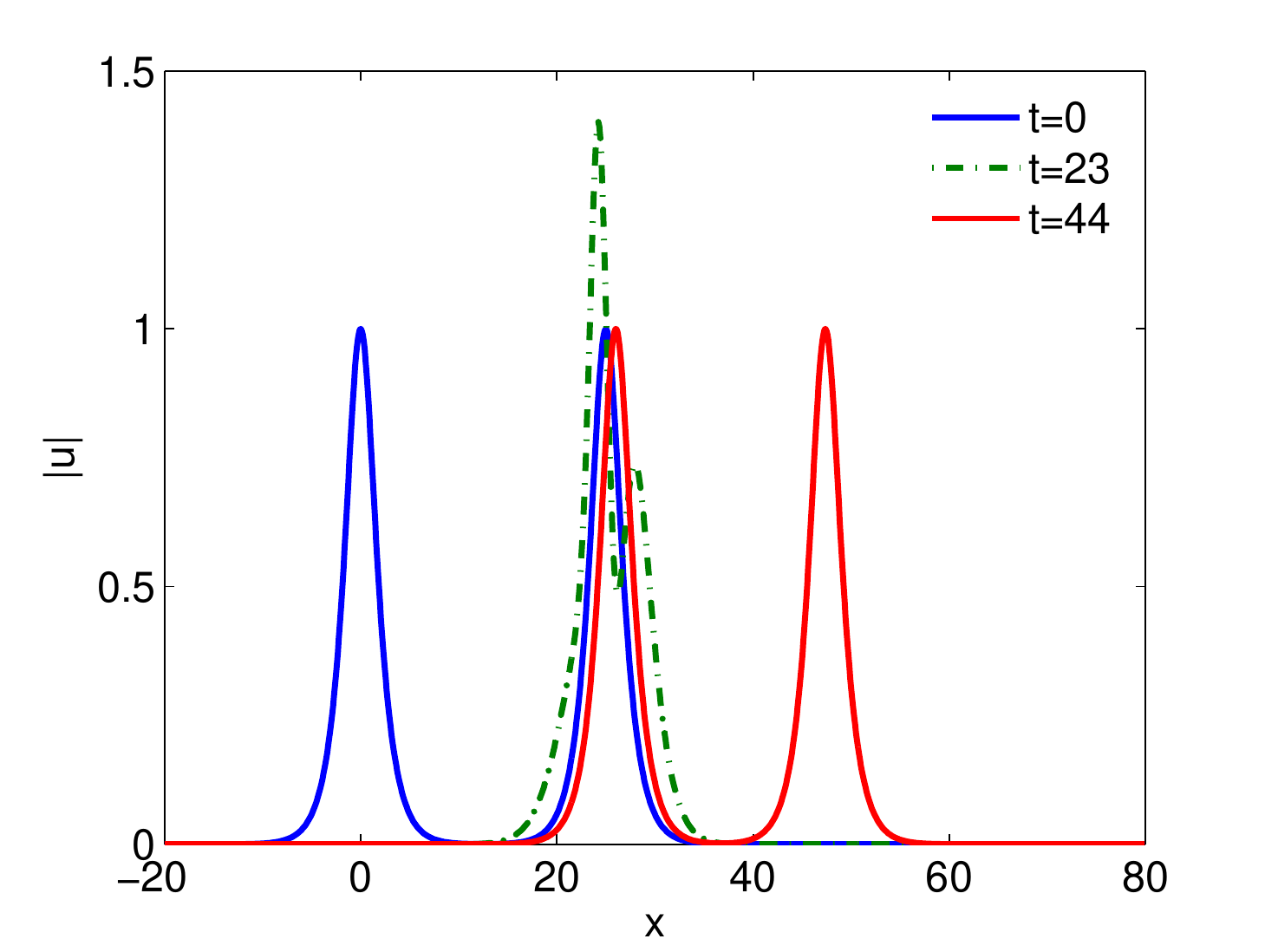}
\caption*{(b)}
\end{minipage}
\caption{ {Time evolution with the soliton collision using the proposed scheme with time step $\tau=0.001$ and Fourier node 1024: (a) the profile of $|u|$ over the time interval $t\in[0,44]$ and (b) snapshots of $|u|$ at times $t=0,\ 23$ and $44$.} }\label{1d-scheme:fig-3}
\end{figure}

\begin{figure}[H]
\centering
\begin{minipage}[t]{70mm}
\includegraphics[width=70mm]{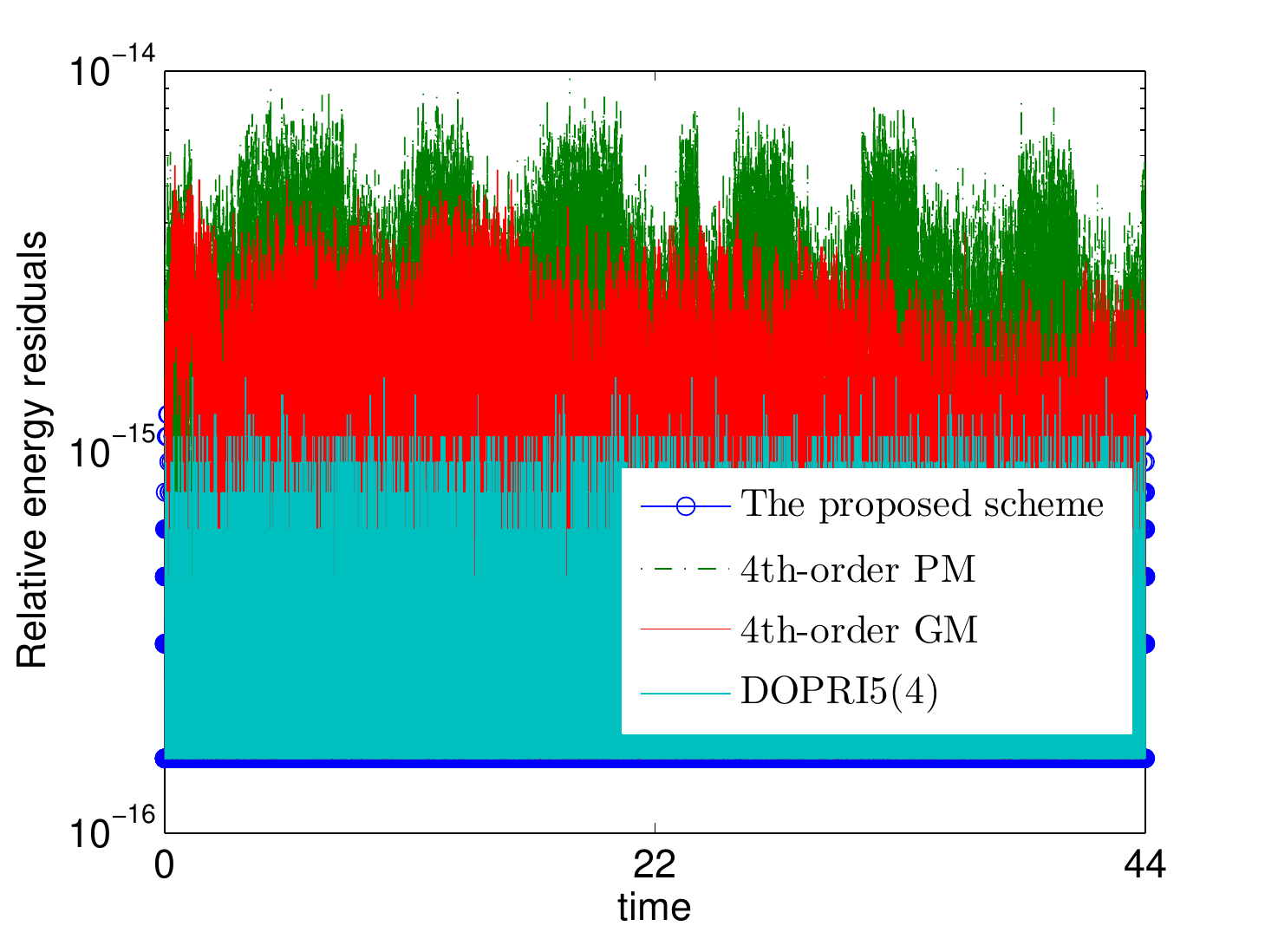}
\end{minipage}
\caption{{The relative energy residuals using the different numerical schemes for the collision of two solitons with time step $\tau=0.001$ and Fourier node 1024 for the one dimensional Schr\"odinger equation \eqref{NLS-equation}.}}\label{1d-scheme:fig:4}
\end{figure}

Finally, we test the two dimensional nonlinear Schr\"odinger equation \eqref{NLS-equation} which possesses the following analytical solution
\begin{align}\label{NLS-AS}
u(x,y,t)=A\exp(\text{i}(k_1x+k_2y-\omega t)),\ \omega=k_1^2+k_2^2-\beta|A|^2.
\end{align}
Here we use the space interval $\Omega=[0,2\pi]^2$ and choose parameters $A=1,\ k_1=k_2=1$ and $\beta=-2$.

To test the temporal discretization errors of the different numerical schemes, we fix the Fourier node $16\times 16$ such that the spatial discretization errors are negligible. The $L^2$ errors and $L^{\infty}$ errors in numerical solution of $u$ in long time (i.e., $t=50$) are calculated using three different numerical schemes with various
time steps {$\tau=0.005, 0.0025, 0.00125$ and $0.000625$}, and the results are displayed in Fig. \ref{scheme:fig:1}. In Fig. \ref{scheme:fig:2}, we show the global $L^2$ errors and $L^{\infty}$ errors of $u$ versus the CPU time using the different numerical schemes at $t=50$. It is clear to see that the observations obtained from  Figs. \ref{scheme:fig:1} and \ref{scheme:fig:2} behave similarly
as those given in Figs. \ref{1d-scheme:fig:1} and \ref{1d-scheme:fig:2}. 

To further investigate the energy-preservation
of the proposed scheme, we provide the energy errors using the different numerical schemes for the two dimensional Schr\"odinger equation over the time interval $t\in[0,1000]$ in Fig. \ref{scheme:fig:3}, which shows that all  methods can exactly preserve the quadratic invariant \eqref{IFD-energy-conservation-law}, which behaves similarly as that given in Fig. \ref{1d-scheme:fig:3}.

\begin{figure}[H]
\centering\begin{minipage}[t]{60mm}
\includegraphics[width=60mm]{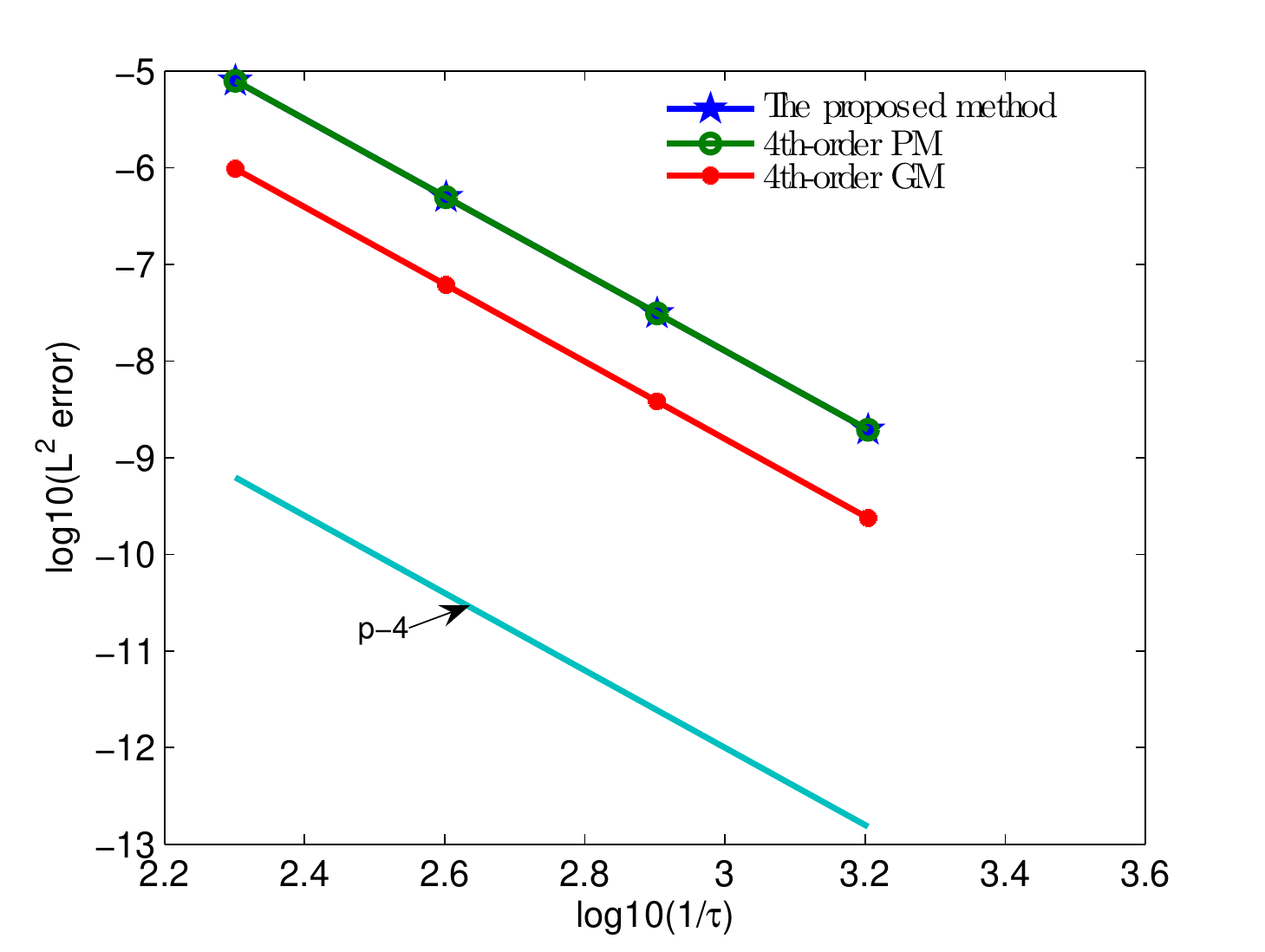}
\end{minipage}
\begin{minipage}[t]{60mm}
\includegraphics[width=60mm]{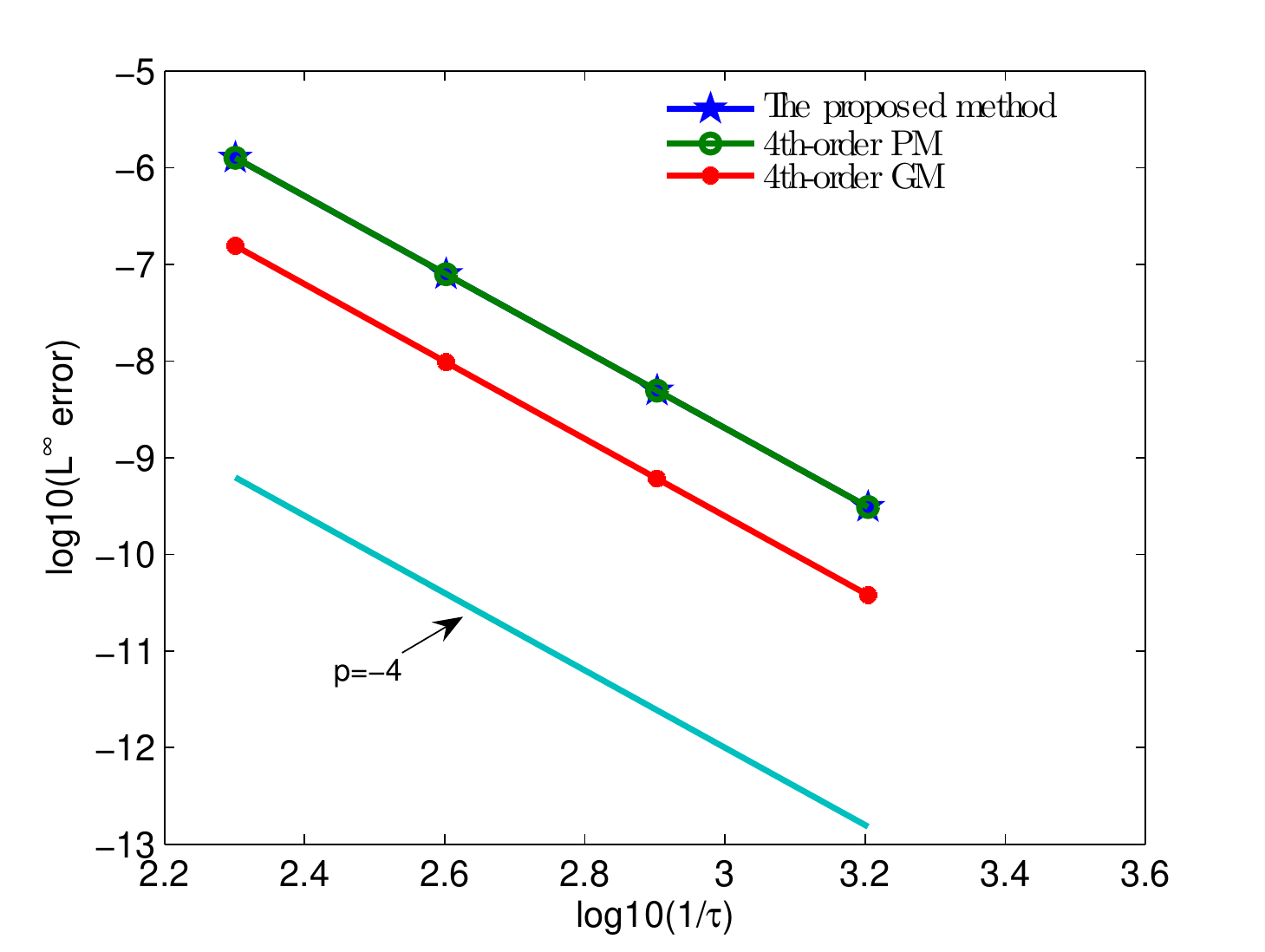}
\end{minipage}
\caption{ Time step refinement tests using the different numerical schemes for the two dimensional Schr\"odinger equation \eqref{NLS-equation}.}\label{scheme:fig:1}
\end{figure}

\begin{figure}[H]
\centering\begin{minipage}[t]{60mm}
\includegraphics[width=60mm]{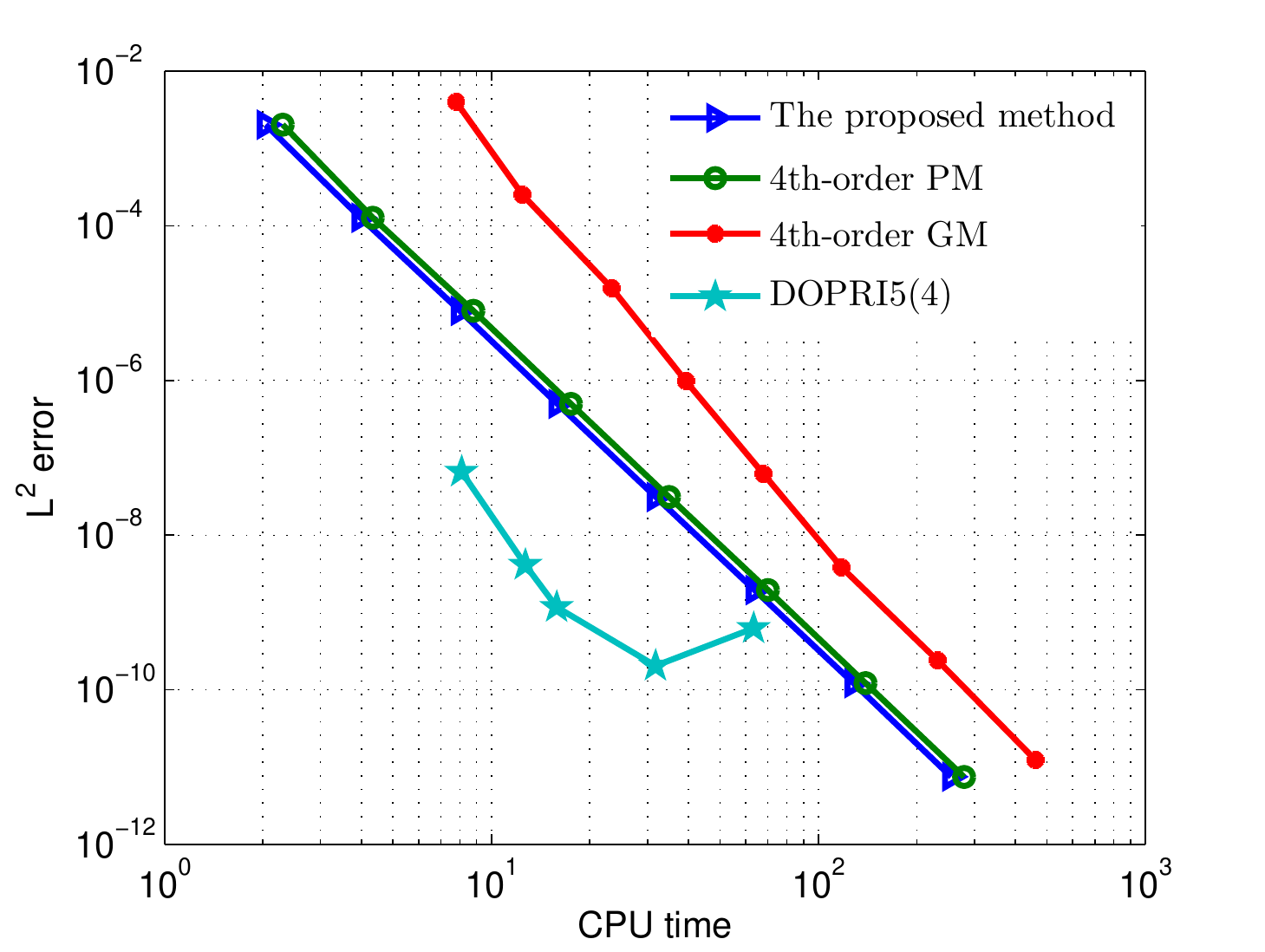}
\end{minipage}
\begin{minipage}[t]{60mm}
\includegraphics[width=60mm]{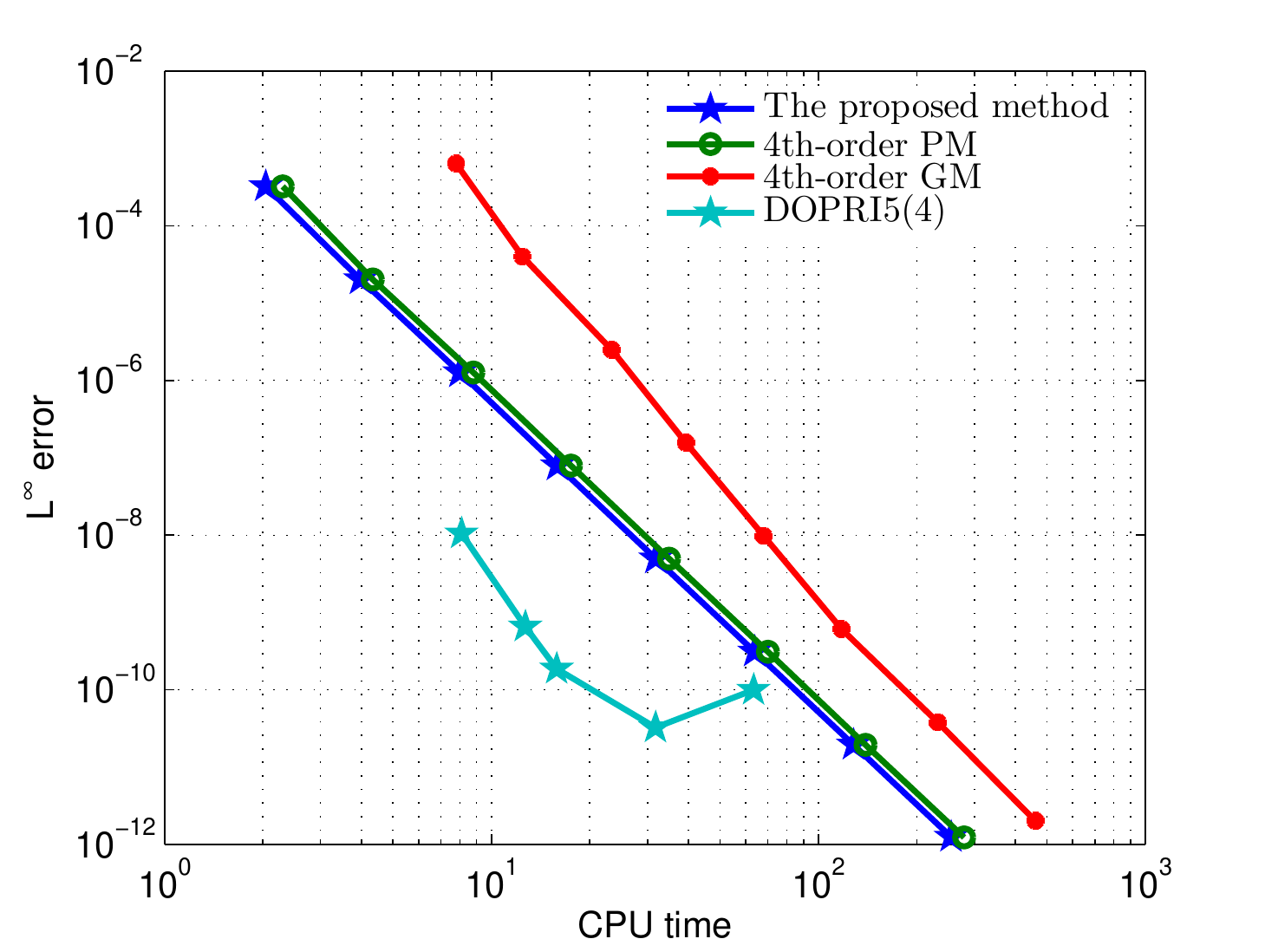}
\end{minipage}
\caption{ The numerical error versus the CPU time using the different numerical schemes for the two dimensional Schr\"odinger equation \eqref{NLS-equation}.}\label{scheme:fig:2}
\end{figure}


\begin{figure}[H]
\centering
\begin{minipage}[t]{70mm}
\includegraphics[width=70mm]{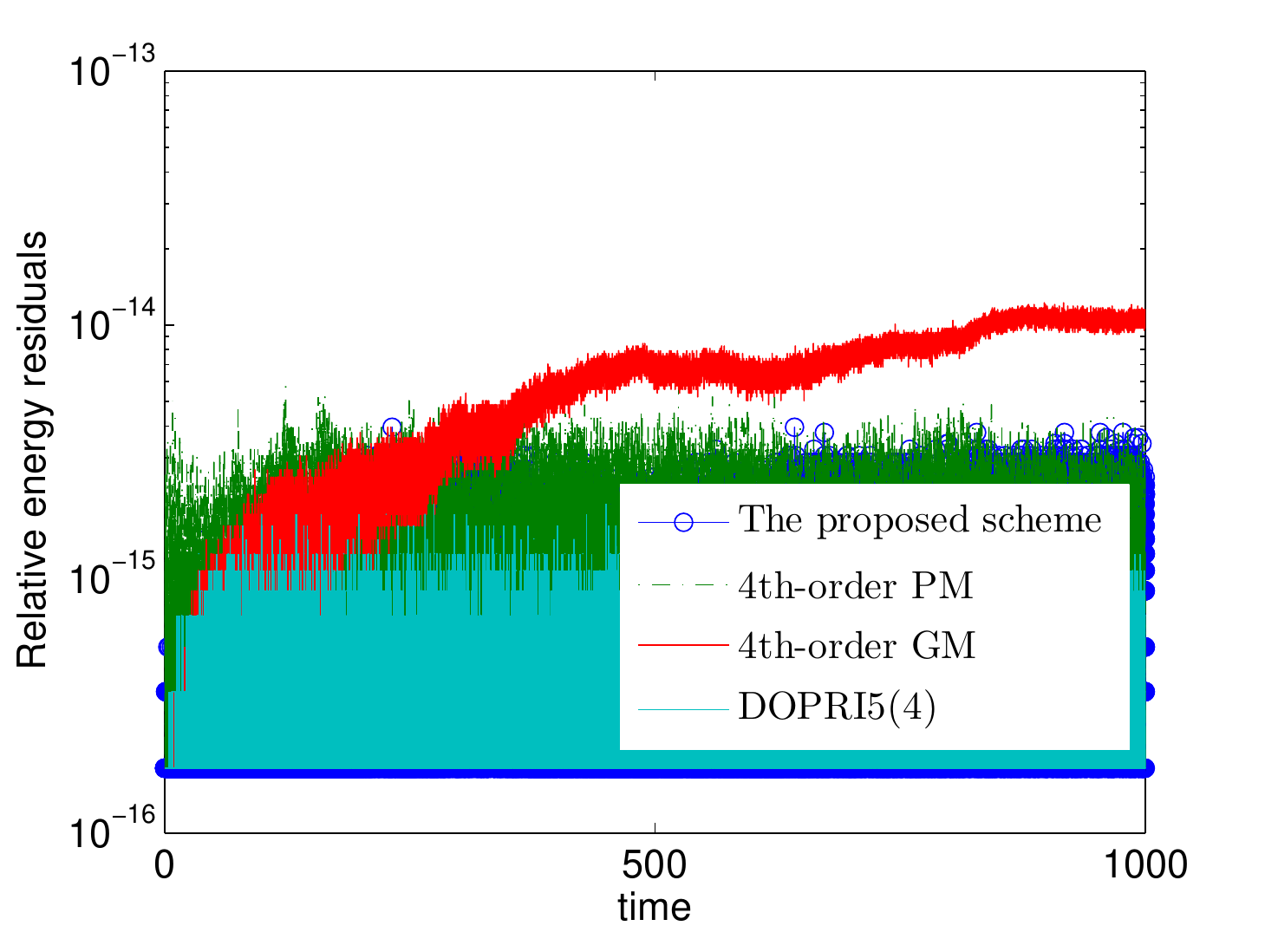}
\end{minipage}
\caption{The relative energy residuals using the different numerical schemes with time step $\tau=0.01$ and Fourier node $16\times 16$ for the two dimensional Schr\"odinger equation \eqref{NLS-equation}.}\label{scheme:fig:3}
\end{figure}

\subsection{Sine-Gordon equation} \label{sec:4.2}
In this subsection, we focus on the sine-Gordon equation given
as follows
\begin{align}\label{sine-gordon-equation}
&\partial_{tt}u-\Delta u+\sin(u)=0,
  \end{align}
where the Hamiltonian energy functional is given by
\begin{align}\label{sine-gordon-equation-Hamiltonian-energy}
\mathcal{H}=\frac{1}{2}\big(\|v\|^2+(u,-\Delta u)+2(1-\cos(u),1)\big).
\end{align}
We let $q=\sqrt{2\Big((1-\cos(u))+\frac{C_0}{|\Omega|}\Big)}$, and rewrite the Hamiltonian energy functional as
 \begin{align}\label{IEQ-energy}
\mathcal{H}=\frac{1}{2}\big(\|v\|^2+(u,-\Delta u)+{\|q\|^2}\big)-C_0.
\end{align}
According to the IEQ reformulation, we obtain the following equivalent system
%
%
 \begin{align}\label{sine-Gordon-eqaution-IEQ-equation}
\left\lbrace
  \begin{aligned}
  &\partial_tu=v,\\
  &\partial_t v=\Delta u-\frac{\sin(u)q}{\sqrt{2\Big((1-\cos(u))+\frac{C_0}{|\Omega|}\Big)}},\\
  &\partial_tq=\frac{\sin(u)\partial_t u}{\sqrt{2\Big((1-\cos(u))+\frac{C_0}{|\Omega|}\Big)}},
  \end{aligned}\right.\ \ 
  \end{align}
  with the consistent initial condition
  \begin{align}\label{sine-Gordon-eqaution-initial}
\left\lbrace
  \begin{aligned}
  &u(x,t=0)=u_0(x),\ {v(x,t=0)=\partial_tu(x,0)},\\
  &q(x,t=0)=\sqrt{2\Big((1-\cos(u_0(x)))+\frac{C_0}{|\Omega|}\Big)}.
  \end{aligned}\right.\ \ 
  \end{align}

Supposing $\Phi=\left[\begin{array}{ccc}
              u \\
              v\\
              q\\
             \end{array}\right]$, then the system \eqref{sine-Gordon-eqaution-IEQ-equation} is equivalently reformulated into the system \eqref{model-general-field} with a quadratic energy
\begin{align}\label{sG-energy-conserved-law}
\mathcal{H}=\frac{1}{2}(\Phi,\mathcal{L}\Phi),\ \mathcal{L}=\left[\begin{array}{ccc}
              -\Delta& \ &\  \\
              \ & 1&\ \\
              \ &\ &1\\
             \end{array}
\right],
\end{align}
and a modified structure matrix
\begin{align*}
\mathcal{G}(\Phi) = \left[\begin{array}{ccc}
0 &1&0\\
-1&0&\frac{-\sin(u)}{\sqrt{2\Big((1-\cos(u))+\frac{C_0}{|\Omega|}\Big)}}\\
0&\frac{\sin(u)}{\sqrt{2\Big((1-\cos(u))+\frac{C_0}{|\Omega|}\Big)}}&0\\
\end{array}
\right].
\end{align*}

Applying \textbf{Scheme \ref{scheme2}} to system \eqref{sine-Gordon-eqaution-IEQ-equation}, we have
\begin{shm}\label{scheme-sG-equation} For given $(u^n,v^n,q^n)$, $(u^{n+1}, v^{n+1},q^{n+1})$ is calculated by the following two steps:
\begin{enumerate}
		\item {Explicit RK: we compute $(\widetilde{u}^{n+1},\widetilde{v}^{n+1},\widetilde{q}^{n+1})$ by using the explicit RK4 method (see remark \ref{rmk:3.5}) to the system \eqref{sine-Gordon-eqaution-IEQ-equation}.}
\item Projection: we update $(u^{n+1},v^{n+1},q^{n+1})$ via
\begin{align}\label{sG-explict-projection}
&u^{n+1}=\tilde{u}^{n+1}+\lambda_n\Big(-\Delta\tilde{u}^{n+1}\Big),\\
  &v^{n+1}=\tilde{v}^{n+1}+\lambda_n\tilde{v}^{n+1},\\
  &q^{n+1}=\tilde{q}^{n+1}+\lambda_n\tilde{q}^{n+1},
\end{align}	
where $\lambda_n$ is given by \eqref{lambda-value}.
\end{enumerate}
\end{shm}

{ Let $U_{j}^n$, $V_{j}^n$ and $Q_{j}^n$ be the numerical approximations of  $u(x_j,t_n)$, $v(x_j,t_n)$ and $q(x_j,t_n)$ for $n=0,1,2\cdots,M$ and $j=0,1,2,\cdots,N$, respectively, and denote $U^n:=[U_0^n,U_1^n,\cdots,U_{N-1}^n]^T$, $V^n:=[V_0^n,V_1^n,\cdots,V_{N-1}^n]^T$ and $Q^n:=[Q_0^n,Q_1^n,\cdots,Q_{N-1}^n]^T$ be the numerical solution vectors. According to Theorem \ref{thm-DECL2}, the {\bf Scheme 4.2} preserves the following discrete energy conservation law 
\begin{align}\label{sg-IFD-energy-conservation-law}
H^{n}=\frac{1}{2}\big(\|V^n\|_h^2+\langle U^n,-\Delta_h U^n\rangle_h+\|Q^n\|_h^2\big)-C_0,
\end{align}
where $\Delta_h$ represents the Fourier pseudo-spectral matrix.}
\begin{rmk} \label{sg-rmk-4.2} {We should note that the energy \eqref{sine-gordon-equation-Hamiltonian-energy} is non-quadratic, thus, according to Remark \ref{rmk2.1}, the {\bf Scheme 4.1} can not preserve the following discrete energy in the original variables
\begin{align}\label{sG-equation-Hamiltonian-energy-error}
H_h^{n}=\frac{1}{2}\big(\|V^n\|_h^2+\langle U^n,-\Delta_h U^n\rangle_h+2\langle 1-\cos(U^n),{\bf1}\rangle_h\big),
\end{align}
where ${\bf 1}=[1,1,\cdots,1]^T\in\mathbb{R}^N$.}
\end{rmk}

We repeat the time step refinement test first and choose the parameter $C_0=1$. The one dimensional sine-Gordon equation \eqref{sine-gordon-equation} admits the analytical solution
\begin{align*}
u(x,t)=4\arctan(t\text{sech}(x)).
\end{align*}
We set the space interval $\Omega=[-50,50]$ with a periodic boundary. The $L^2$ errors and $L^{\infty}$ errors of $u$ { at time $t=10$} using the different numerical schemes, Fourier node 1024 and various time steps { $\tau=0.01, 0.005, 0.0025,$ and $0.00125$ }are showed in Fig. \ref{1d-sG-scheme-error}.  Moreover, we plot the global $L^2$ errors and $L^{\infty}$ errors of $u$ versus the the CPU time {at time $t=10$} using the different schemes with Fourier node 1024 in Fig. \ref{1d-sG-scheme-CPU}. Here we observe similar results, i.e., all methods have fourth order accuracy in time errors , the error provided by the 4th-order GM is smallest, and our scheme has the same order of magnitude as the one of the 4th-order PM. Analogously, for a given global error, DOPRI5(4) is computationally cheapest and the 4th-order GM is most expensive.

In Fig. \ref{scheme:fig:6}, we display the errors in the discrete  invariants over the time interval $t\in[0,1000]$, which shows that the DOPRI5(4), the 4th-order GM, 4th-order PM and the proposed scheme can preserve  the discrete modified energy \eqref{sg-IFD-energy-conservation-law}, and our scheme can only preserve the discrete Hamiltonian energy \eqref{sG-equation-Hamiltonian-energy-error} approximately, which is consistent with Remark \ref{sg-rmk-4.2}.



\begin{figure}[H]
\centering\begin{minipage}[t]{60mm}
\includegraphics[width=60mm]{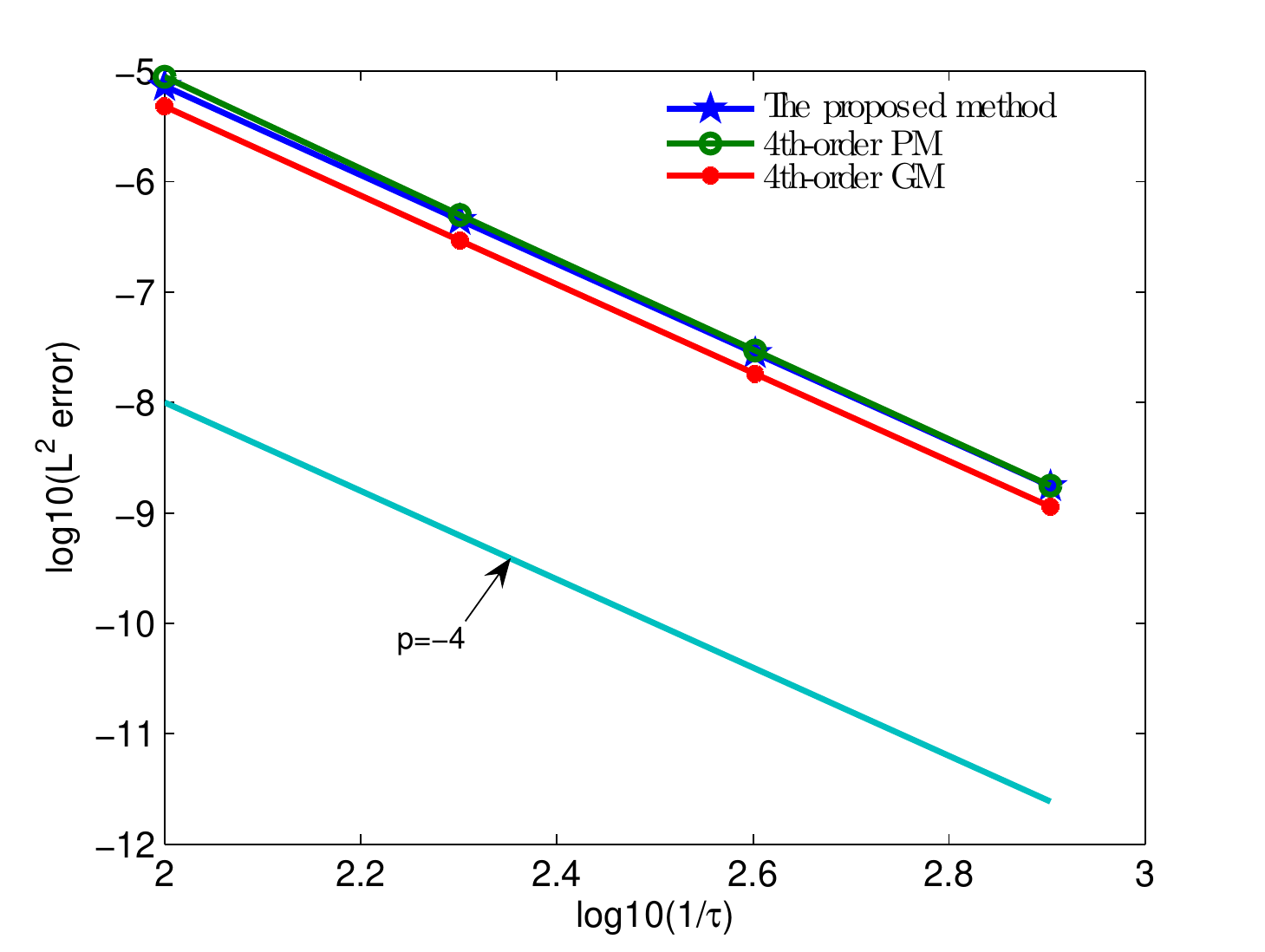}
\end{minipage}
\begin{minipage}[t]{60mm}
\includegraphics[width=60mm]{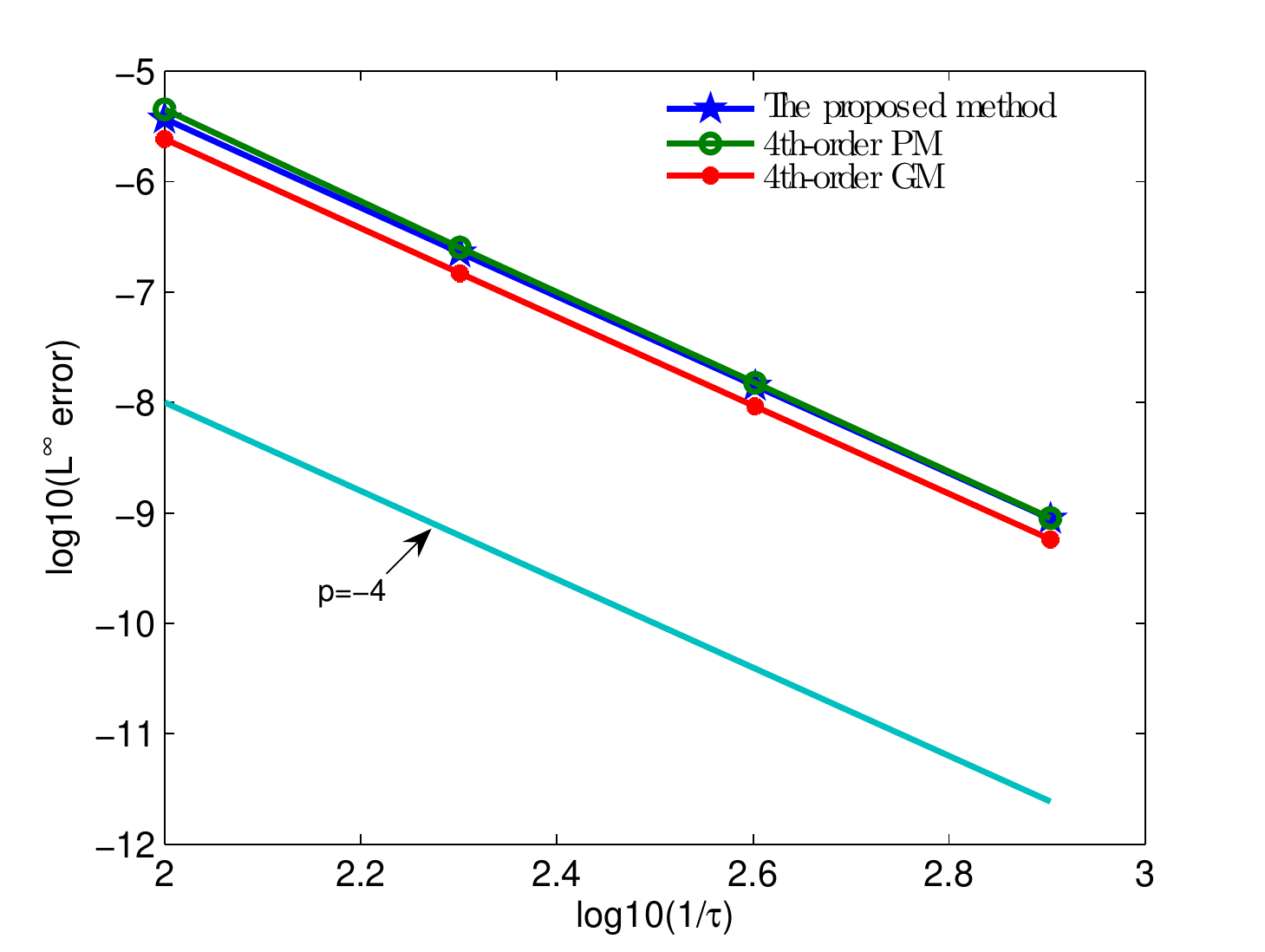}
\end{minipage}
\caption{ Time step refinement tests using the different numerical schemes for the one dimensional
sine-Gordon equation \eqref{sine-gordon-equation}.}\label{1d-sG-scheme-error}
\end{figure}

\begin{figure}[H]
\centering\begin{minipage}[t]{60mm}
\includegraphics[width=60mm]{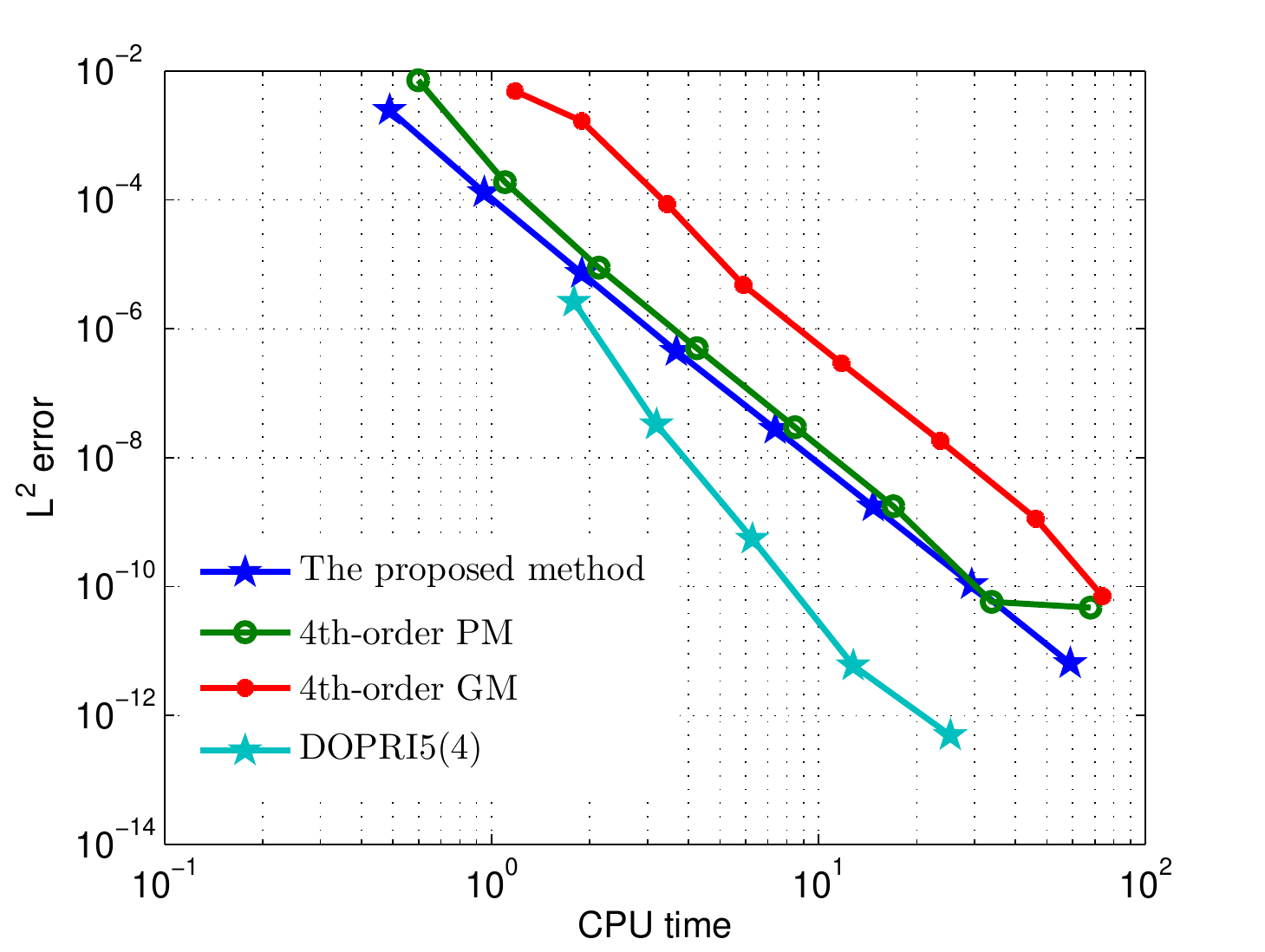}
\end{minipage}
\begin{minipage}[t]{60mm}
\includegraphics[width=60mm]{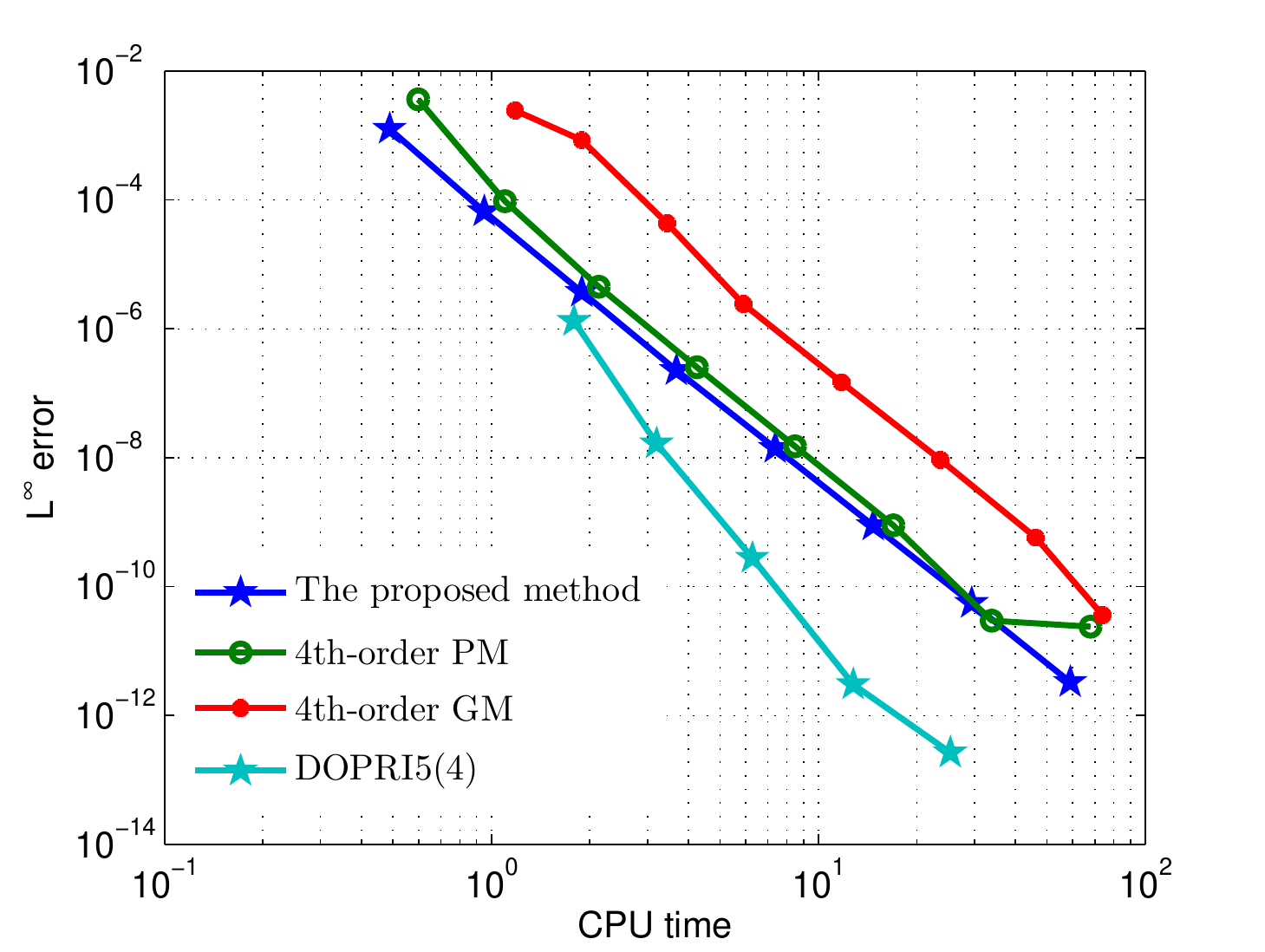}
\end{minipage}
\caption{The numerical error versus the CPU time using the different numerical schemes for
one dimensional sine-Gordon equation \eqref{sine-gordon-equation}.}\label{1d-sG-scheme-CPU}
\end{figure}

\begin{figure}[H]
\centering
\begin{minipage}[t]{70mm}
\includegraphics[width=70mm]{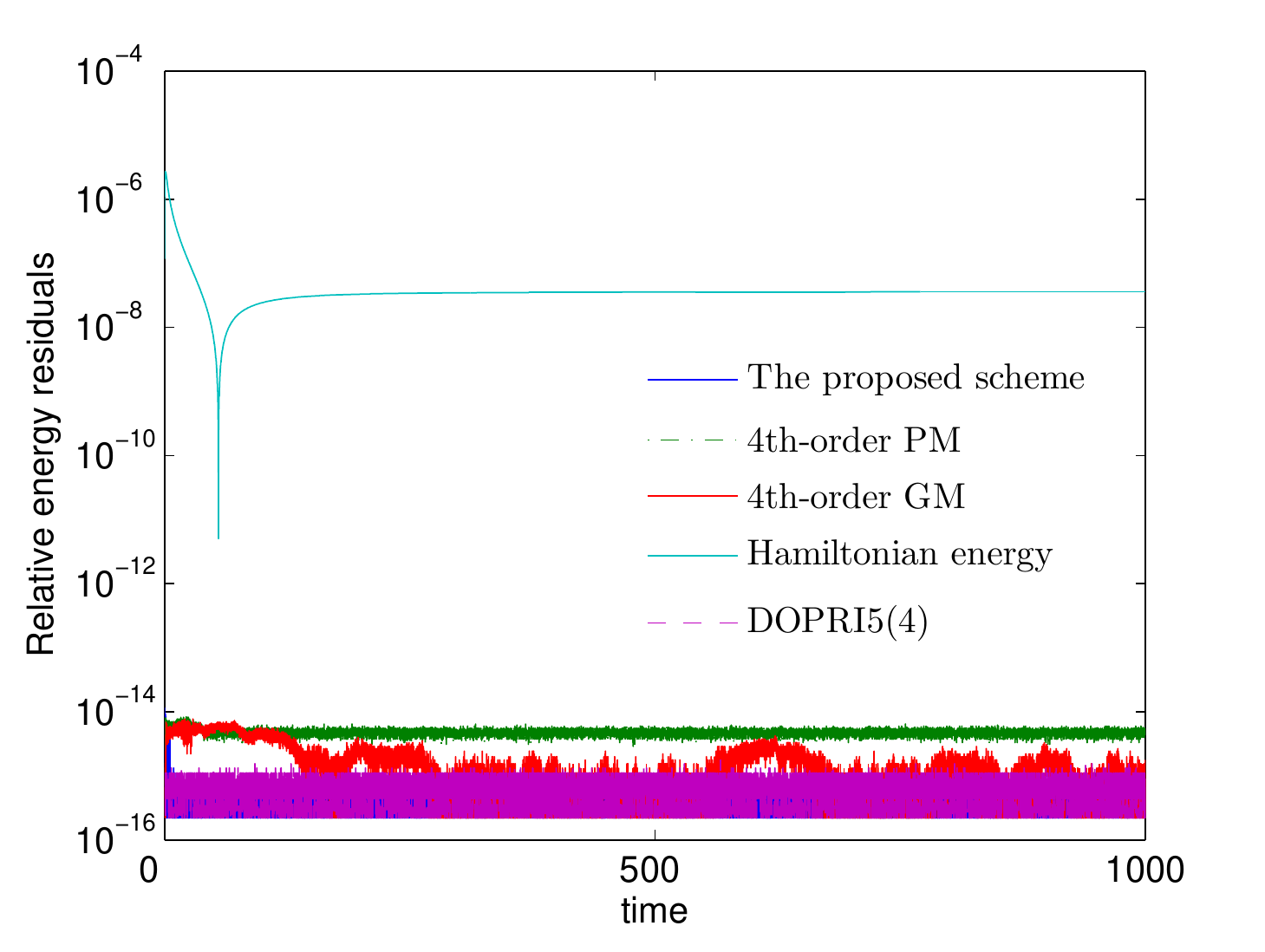}
\end{minipage}
\caption{The relative energy residuals including the discrete quadratic energy \eqref{sg-IFD-energy-conservation-law} using the different numerical schemes and the discrete Hamiltonian energy \eqref{sG-equation-Hamiltonian-energy-error} provided by the proposed scheme with the time step $\tau=0.02$ and Fourier node 256 for
one dimensional sine-Gordon equation \eqref{sine-gordon-equation}.}\label{scheme:fig:6}
\end{figure}

Next, we apply the proposed scheme to simulate the collision of four ring solitons for the two dimensional sine-Gordon equation with initial
conditions given as follows \cite{SKV10} 
\begin{align*}
&u(x,y,0)=4\tan^{-1}\left[\exp\left(\frac{4-\sqrt{(x+3)^{2}+(y+7)^{2}}}{0.436}\right)\right],\\
& u_t(x,y,0)=\frac{4.13}{\cosh\left(\frac{4-\sqrt{(x+3)^{2}+(y+7)^{2}}}{0.436}\right)},\ (x,y)\in\Omega.
\end{align*}
We choose the space interval $\Omega=[-30,10]^2$ and a periodic boundary condition.
The collision precisely among four expanding circular ring solitons are summarized in Fig. \ref{scheme:fig:7}, showing a strong agreement with the existing results presented in Refs. \cite{CJWS19jcp,SKV10}. Here, one should notice that the numerical solutions includes the extension across $x=-10$ and $y=-10$ by symmetry properties of the problem. The errors including the quadratic energy \eqref{sg-IFD-energy-conservation-law} calculated
with different schemes and the discrete Hamiltonian energy \eqref{sG-equation-Hamiltonian-energy-error} provided by the proposed scheme over the time interval $t\in[0,100]$ are shown in Fig. \ref{scheme:fig:8}, which behaves similarly as that given in Fig. \ref{scheme:fig:6}.

\begin{figure}[H]
\centering
\begin{minipage}[t]{60mm}
\includegraphics[width=60mm]{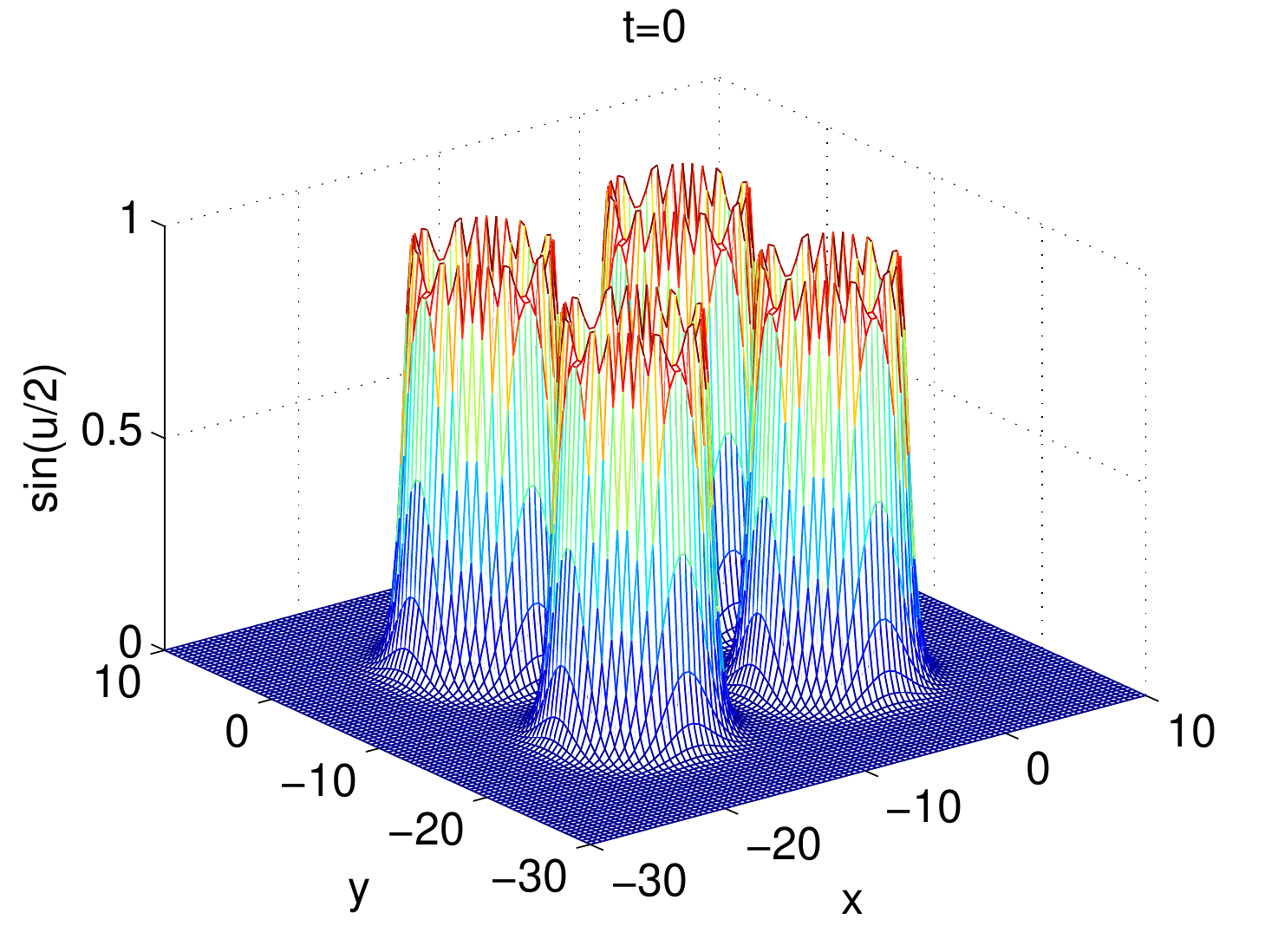}
\end{minipage}
\begin{minipage}[t]{60mm}
\includegraphics[width=60mm]{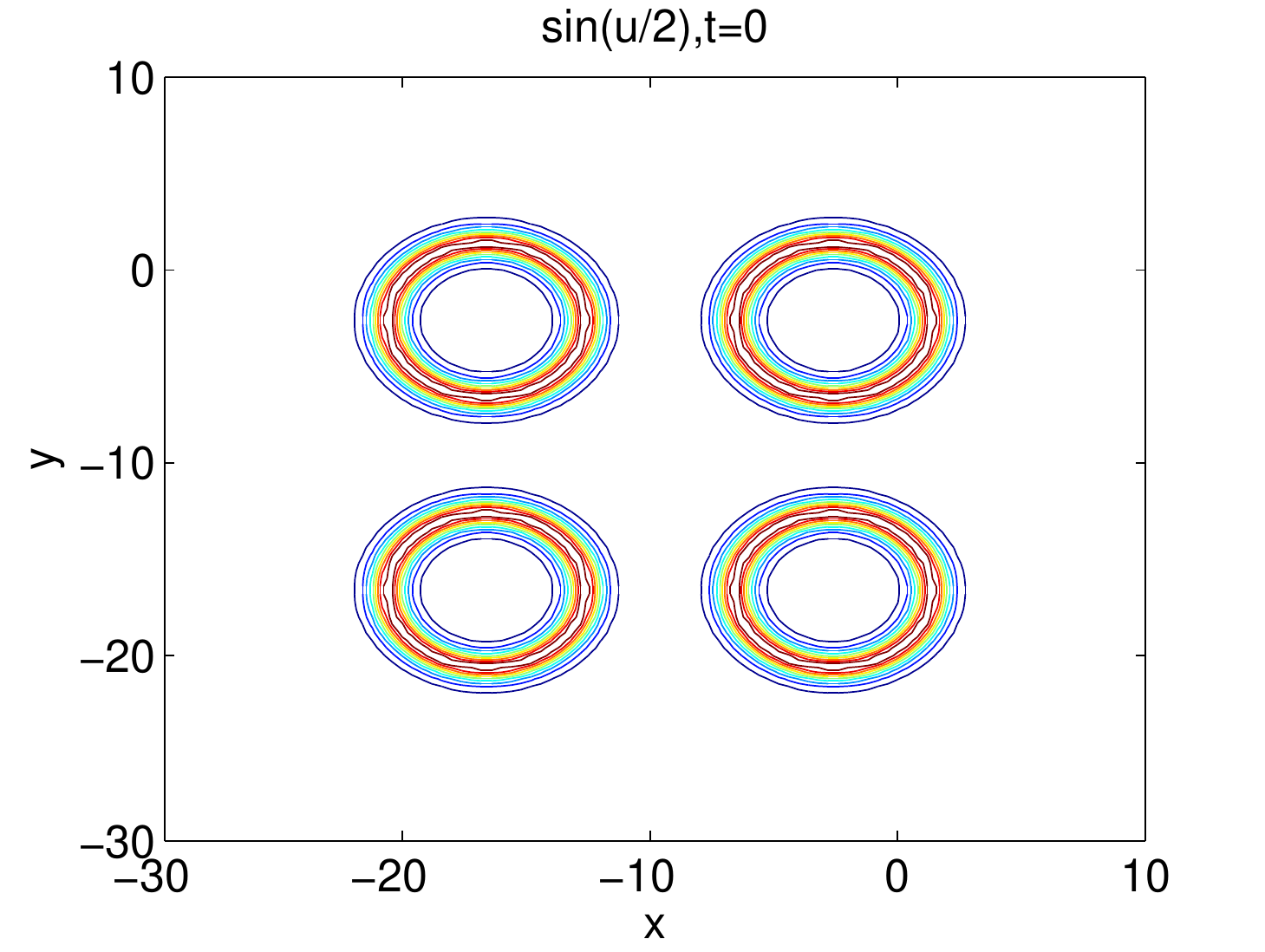}
\end{minipage}
\centering
\begin{minipage}[t]{60mm}
\includegraphics[width=60mm]{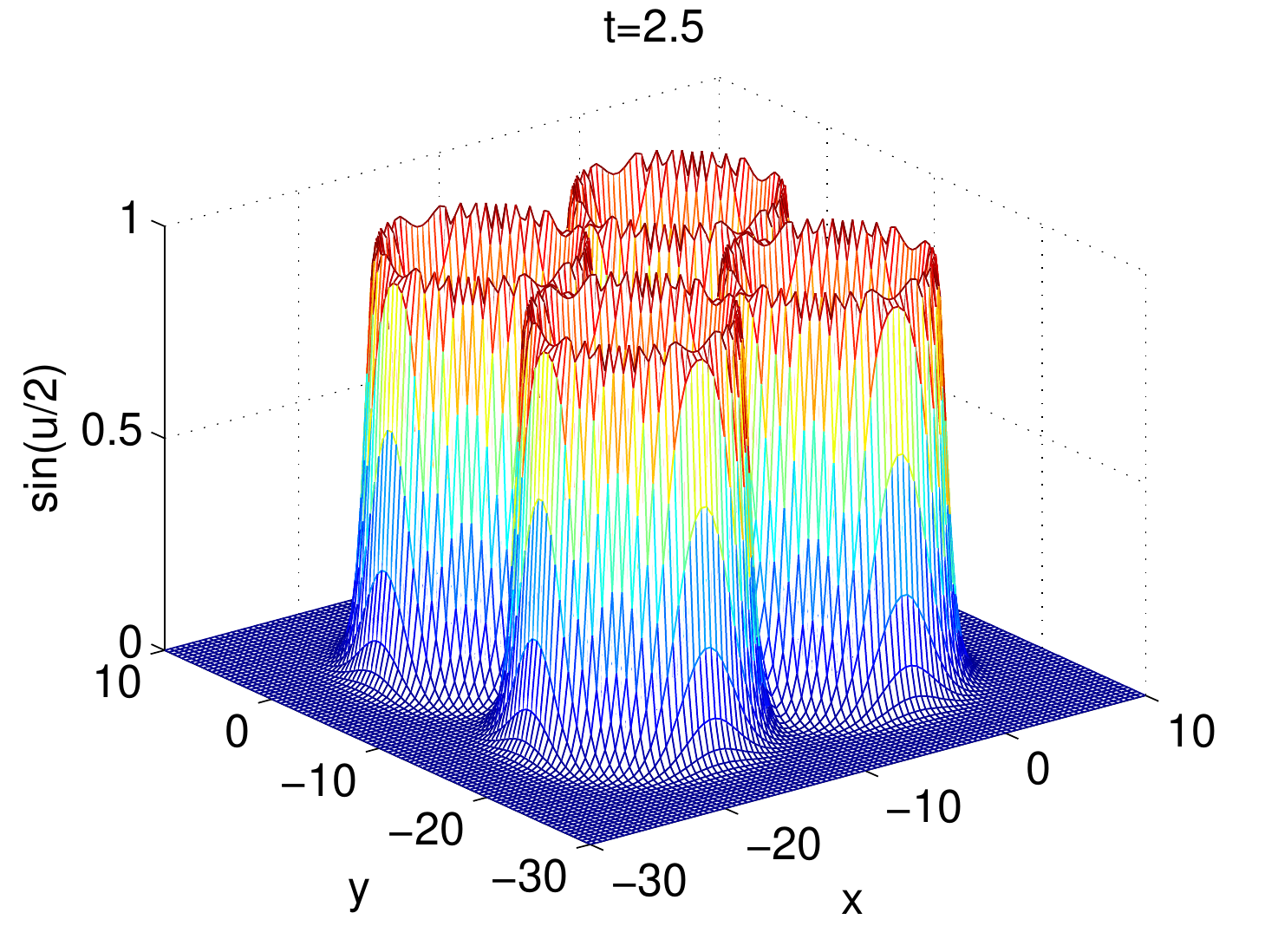}
\end{minipage}
\begin{minipage}[t]{60mm}
\includegraphics[width=60mm]{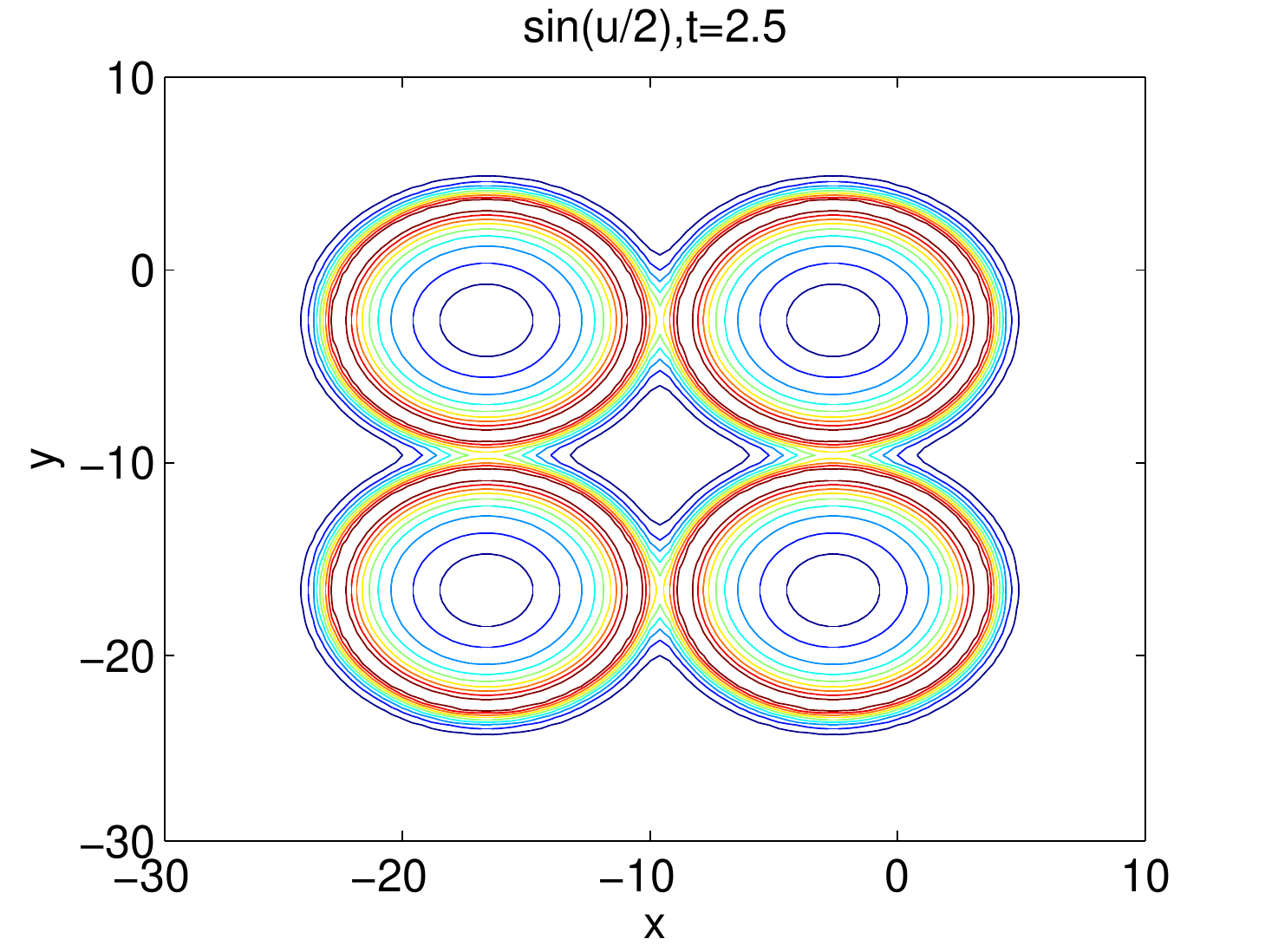}
\end{minipage}
\centering
\begin{minipage}[t]{60mm}
\includegraphics[width=60mm]{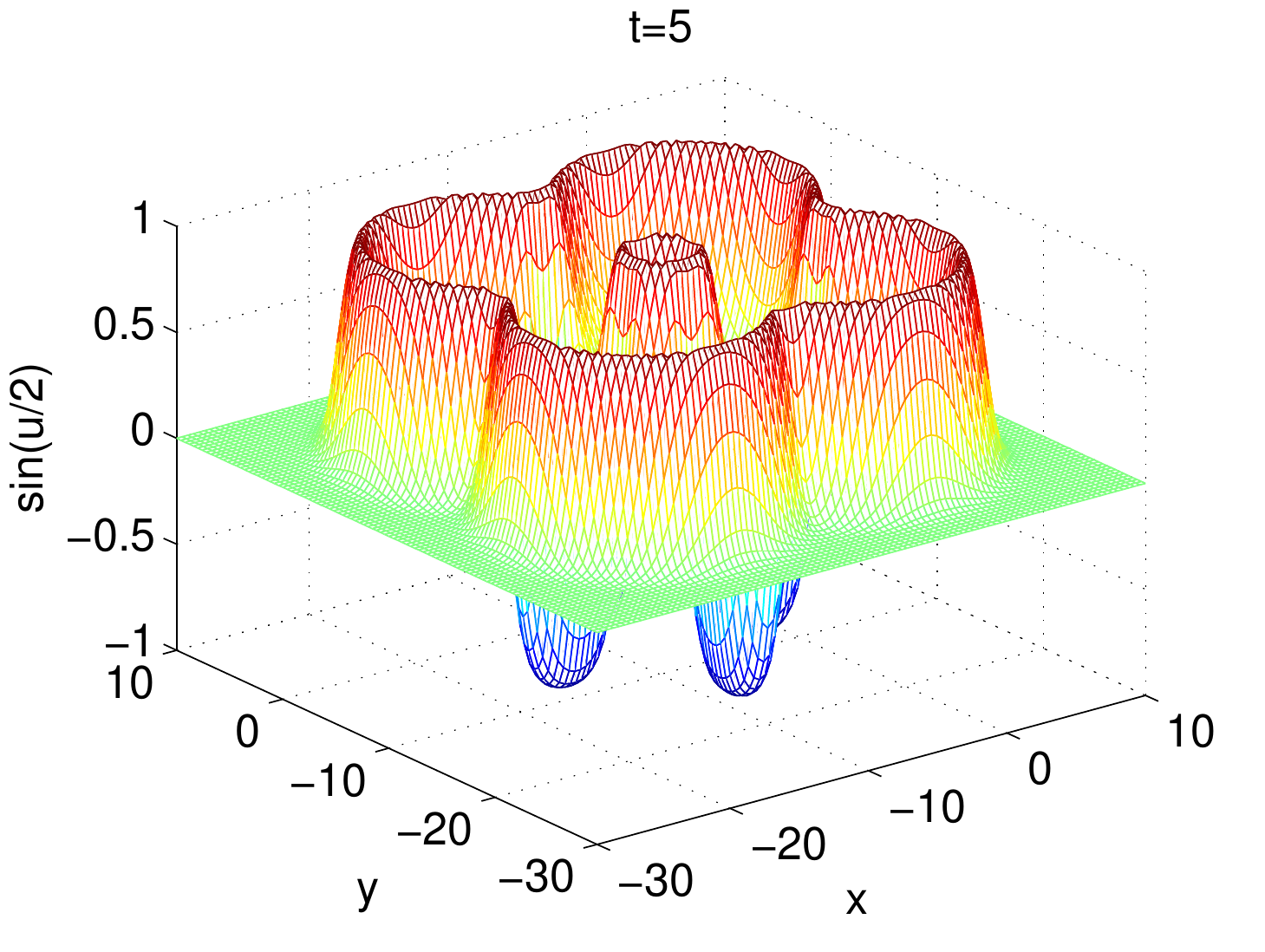}
\end{minipage}
\begin{minipage}[t]{60mm}
\includegraphics[width=60mm]{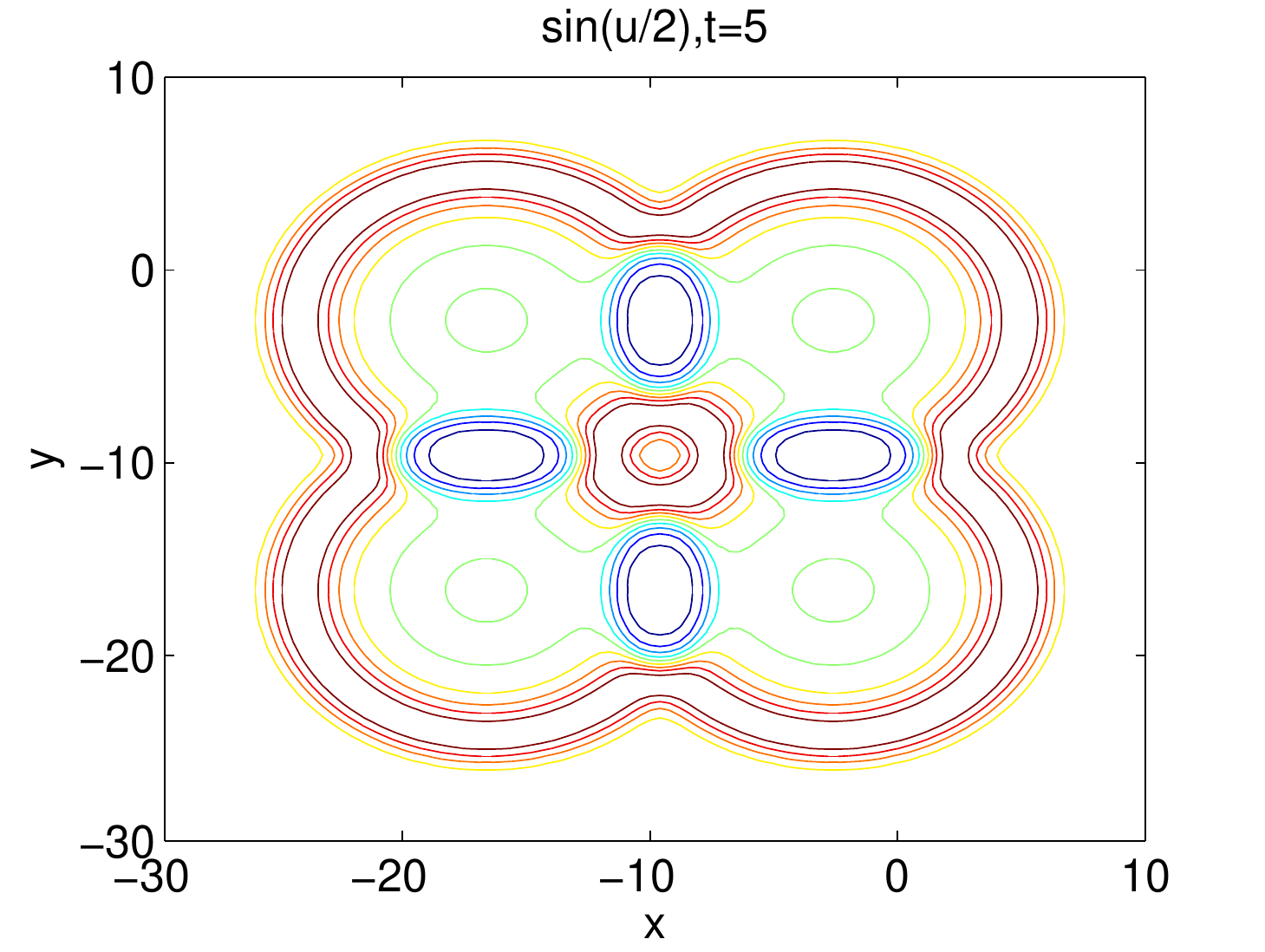}
\end{minipage}
\centering
\begin{minipage}[t]{60mm}
\includegraphics[width=60mm]{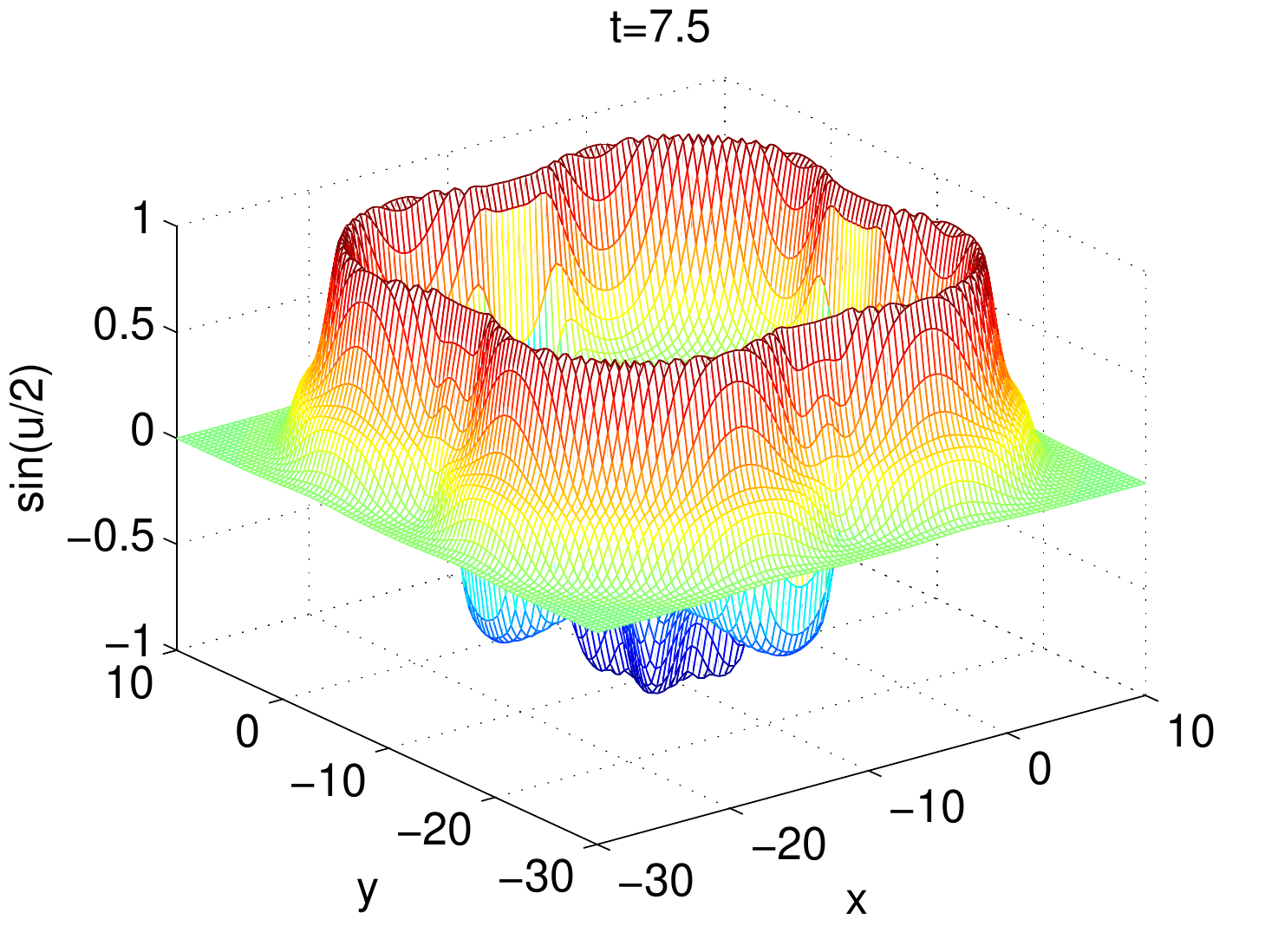}
\end{minipage}
\begin{minipage}[t]{60mm}
\includegraphics[width=60mm]{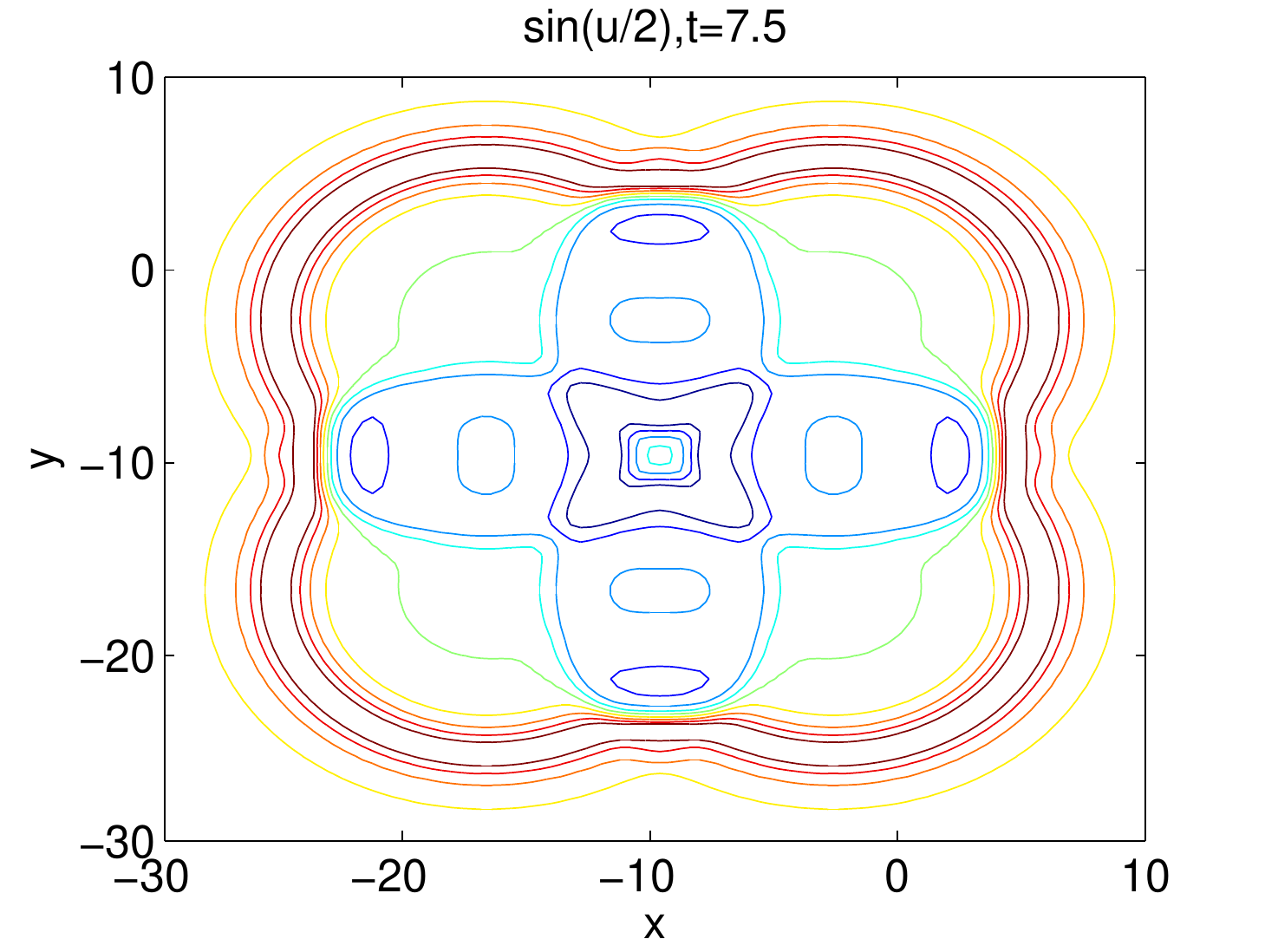}
\end{minipage}
\centering
\begin{minipage}[t]{60mm}
\includegraphics[width=60mm]{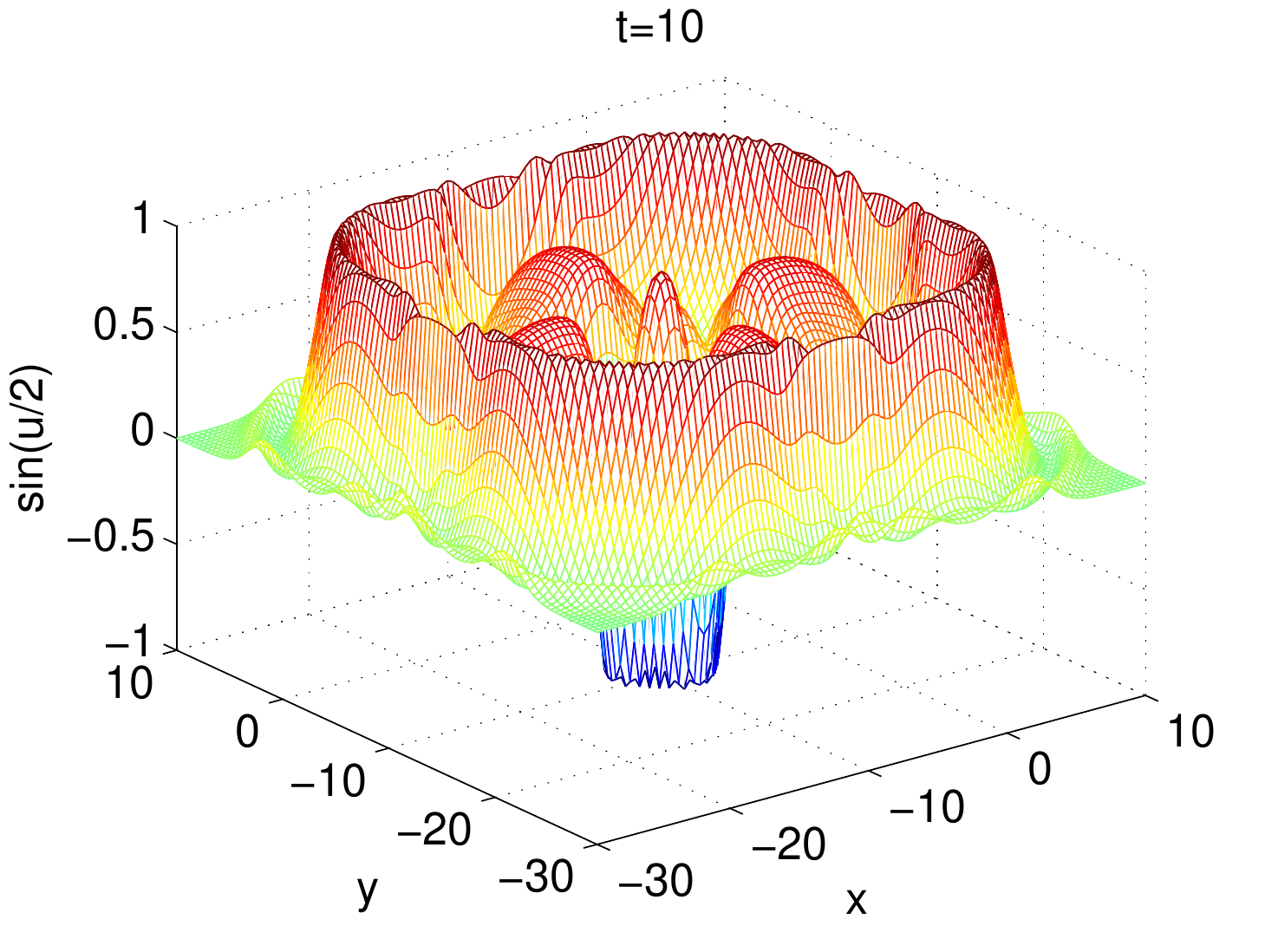}
\end{minipage}
\begin{minipage}[t]{60mm}
\includegraphics[width=60mm]{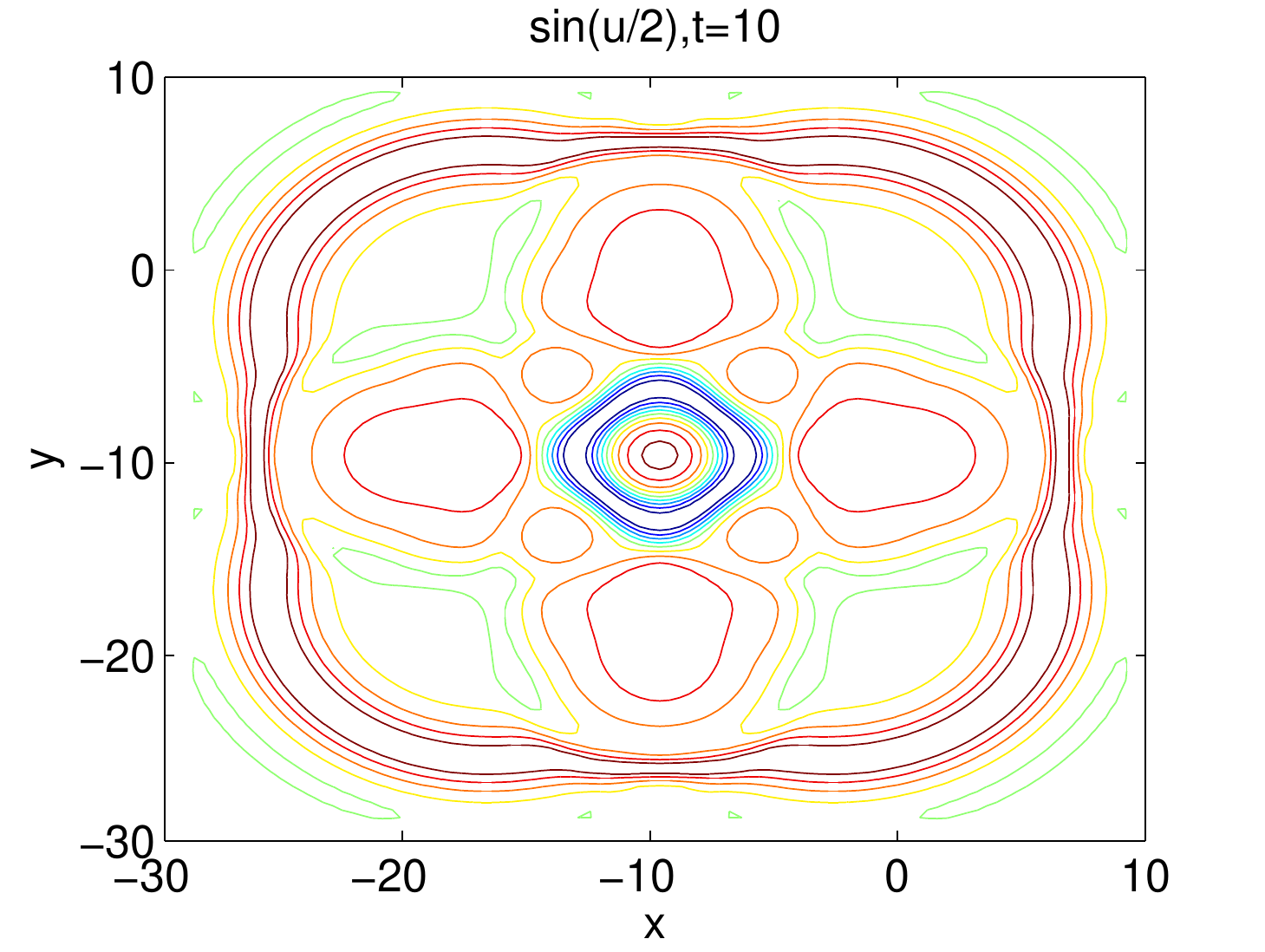}
\end{minipage}
\caption{The profile of $\sin(u/2)$ at $t=0,2.5,5,7.5$ and $10$ using time step $\tau=0.1$ and Fourier node $200\times 200$ with the proposed scheme for
two dimensional sine-Gordon equation \eqref{sine-gordon-equation}.}\label{scheme:fig:7}
\end{figure}

\begin{figure}[H]
\centering
\begin{minipage}[t]{70mm}
\includegraphics[width=70mm]{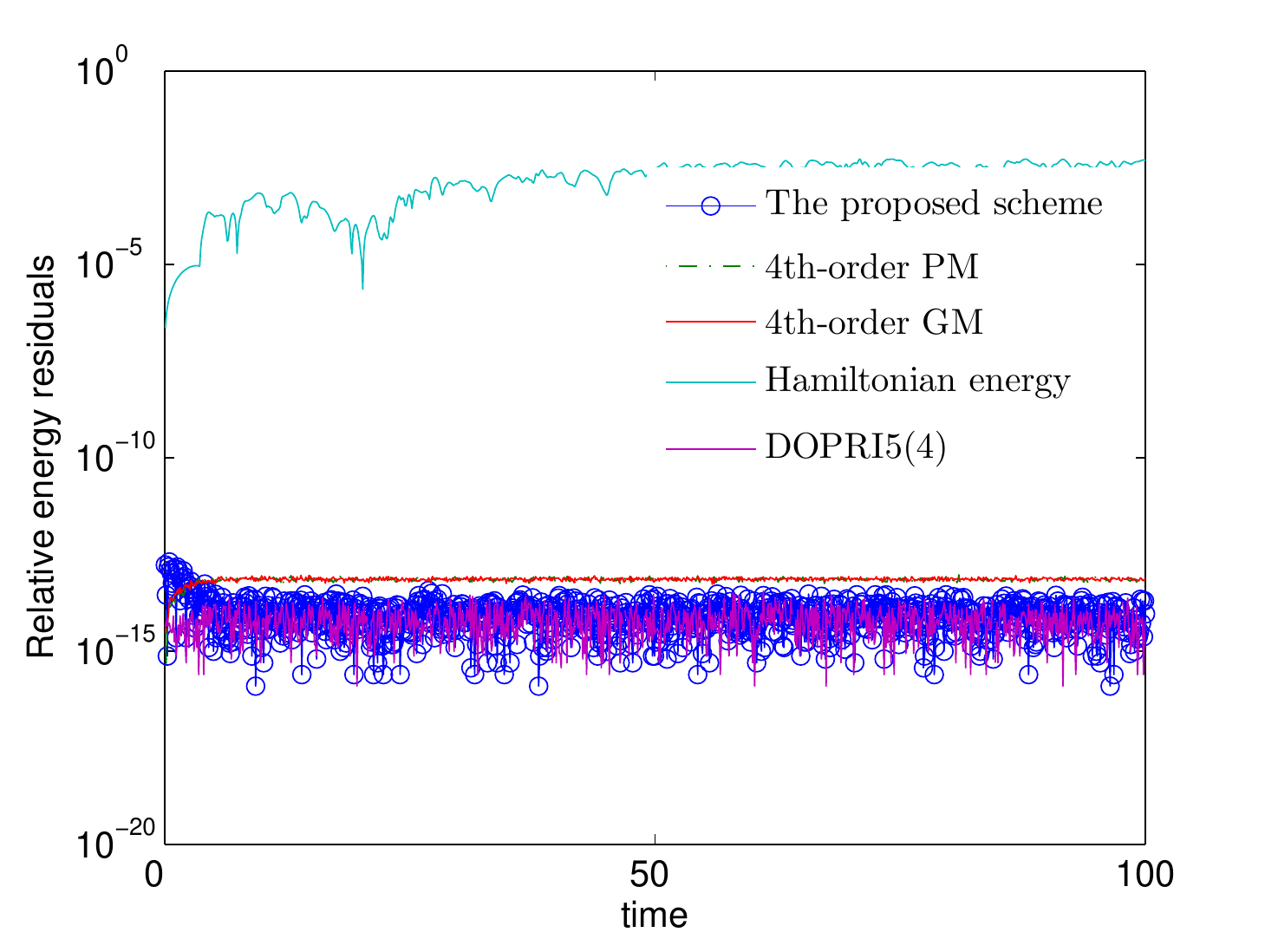}
\end{minipage}
\caption{The relative energy residuals including the discrete quadratic energy \eqref{sg-IFD-energy-conservation-law} using the four numerical schemes and the discrete Hamiltonian energy \eqref{sG-equation-Hamiltonian-energy-error} provided by the proposed scheme with time step $\tau=0.1$ and Fourier node $200\times 200$ for
two dimensional sine-Gordon equation \eqref{sine-gordon-equation}.}\label{scheme:fig:8}
\end{figure}

\section{Concluding remarks}\label{Sec:PM:7}
In this paper, we have presented a new systematic and unified way to develop explicit high-order energy-preserving methods for general Hamiltonian PDEs by combining the orthogonal projection method with the explicit RK methods. Numerical examples
are addressed to illustrate the accuracy, CPU time and invariants-preservation of the proposed
methods. Compared with the two existing energy-preserving schemes of same order, the proposed high-order schemes show remarkable efficiency.

We conclude this paper with some remarks. First, compared with the incremental direction method proposed in Ref. \cite{CHMR06}, the proposed method required to  evaluate the gradient of the energy, and can not preserve linear invariants and affine invariants, however, for the quadratic invariant, the analytical expression of the gradient for the invariant is easily obtained and numerical results show the computation cost of our method is much cheaper. In addition, the expression of the Lagrange multiplier $\lambda_n$ (see \eqref{lambda-value}) is more concise. Second, the proposed method might not work well for highly unstable systems, thus, implicit projections (e.g., see Ref. \cite{BIT12siam}) become necessary. However, such trade-offs among methods should be further investigated. Finally, for several invariants, the Lagrange multiplier of the proposed projection can not be explicitly obtained. Thus, a possible future work is to develop explicit high-order multiple invariants-preserving methods.

\section*{Acknowledgments}
Chaolong Jiang's work is partially supported by the National Natural Science Foundation of China (Grant No. 11901513), the Yunnan Provincial Department of Education Science Research Fund Project (Grant No. 2019J0956) and the Science and Technology Innovation Team on Applied Mathematics in Universities of Yunnan. Yushun Wang's work is partially supported by the National Natural Science Foundation of China (Grant No. 11771213). Yuezheng Gong's work is partially supported by the Natural Science Foundation of Jiangsu Province
(Grant No. BK20180413), the National Natural Science Foundation of China (Grant No. 11801269) and the Foundation of Jiangsu Key Laboratory for Numerical Simulation of Large Scale Complex
Systems (Grant No. 202002).

\bibliographystyle{plain}

\begin{thebibliography}{10}

\bibitem{BCMR12}
L.~Brugnano, M.~Calvo, J.~I. Montijano, and L.~R\'andez.
\newblock Energy-preserving methods for {P}oisson systems.
\newblock {\em J. Comput. Appl. Math.}, 236:3890--3904, 2012.

\bibitem{BI16}
L.~Brugnano and F.~Iavernaro.
\newblock {\em Line Integral Methods for Conservative Problems}.
\newblock Chapman et Hall/CRC: Boca Raton, FL, USA, 2016.

\bibitem{BIT10}
L.~Brugnano, F.~Iavernaro, and D.~Trigiante.
\newblock Hamiltonian boundary value methods (energy preserving discrete line
  integral methods).
\newblock {\em J. Numer. Anal. Ind. Appl. Math.}, 5:17--37, 2010.

\bibitem{BITjcam11}
L.~Brugnano, F.~Iavernaro, and D.~Trigiante.
\newblock A note on the efficient implementation of {H}amiltonian {BVM}s.
\newblock {\em J. Comput. Appl. Math.}, 236:375--383, 2011.

\bibitem{BIT12siam}
L.~Brugnano, F.~Iavernaro, and D.~Trigiante.
\newblock Energy-and quadratic invariants-preserving integrators based upon
  {G}auss collocation formulae.
\newblock {\em SIAM J. Numer. Anal.}, 50:2897--2916, 2012.

\bibitem{BM02jcam}
N.~Del Buono and C.~Mastroserio.
\newblock Explicit methods based on a class of four stage fourth order
  {R}unge-{K}utta methods for preserving quadratic laws.
\newblock {\em J. Comput. Appl. Math.}, 140:231--243, 2002.

\bibitem{CS19jcp}
J.~Cai and J.~Shen.
\newblock Two classes of linearly implicit local energy-preserving approach for
  general multi-symplectic {H}amiltonian {PDE}s.
\newblock {\em J. Comput. Phys.}, 401:108975, 2020.

\bibitem{CJWS19jcp}
W.~Cai, C.~Jiang, Y.~Wang, and Y.~Song.
\newblock Structure-preserving algorithms for the two-dimensional sine-{G}ordon
  equation with {N}eumann boundary conditions.
\newblock {\em J. Comput. Phys.}, 395:166--185, 2019.

\bibitem{CHMR06}
M.~Calvo, D.~Hern\'andez-Abreu, J.~I. Montijano, and L.~R\'andez.
\newblock On the preservation of invariants by explicit {R}unge-{K}utta
  methods.
\newblock {\em SIAM J. Sci. Comput.}, 28:868--885, 2006.

\bibitem{CIZ97}
M.~Calvo, A.~Iserles, and A.~Zanna.
\newblock Numerical solution of isospectral flows.
\newblock {\em Math. Comp.}, 66:1461--1486, 1997.

\bibitem{CLMR15}
M.~Calvo, M.P. Laburta, J.I. Montijano, and L.~R\'andez.
\newblock Runge-kutta projection methods with low dispersion and dissipation
  errors.
\newblock {\em Adv. Comput. Math.}, 41:231--251, 2015.

\bibitem{CGM12}
E.~Celledoni, V.~Grimm, R.~I. McLachlan, D.~I. McLaren, D.~O'Neale, B.~Owren,
  and G.~R.~W. Quispel.
\newblock Preserving energy resp. dissipation in numerical {PDE}s using the
  ``{A}verage {V}ector {F}ield" method.
\newblock {\em J. Comput. Phys.}, 231:6770--6789, 2012.

\bibitem{CMMOQWesiam09}
E.~Celledoni, R.~I. McLachlan, D.~I. McLaren, B.~Owren, G.~R.~W. Quispel, and
  W.~M. Wright.
\newblock Energy-preserving {R}unge-{K}utta methods.
\newblock {\em ESAIM: M2AN}, 43:645--649, 2009.

\bibitem{CQ01}
J.~Chen and M.~Qin.
\newblock Multi-symplectic {F}ourier pseudospectral method for the nonlinear
  {S}chr\"{o}dinger equation.
\newblock {\em Electr. Trans. Numer. Anal.}, 12:193--204, 2001.

\bibitem{CH11bit}
D.~Cohen and E.~Hairer.
\newblock Linear energy-preserving integrators for {P}oisson systems.
\newblock {\em BIT}, 51:91--101, 2011.

\bibitem{Cooper87}
G.~J. Cooper.
\newblock Stability of {R}unge-{K}utta methods for trajectory problems.
\newblock {\em IMA J. Numer. Anal.}, 7:1--13, 1987.

\bibitem{DO11}
M.~Dahlby and B.~Owren.
\newblock A general framework for deriving integral preserving numerical
  methods for {PDE}s.
\newblock {\em SIAM J. Sci. Comput.}, 33:2318--2340, 2011.

\bibitem{DP80jcam}
J.~R. Dormand and P.~J. Prince.
\newblock A family of embedded {R}unge-{K}utta formulae.
\newblock {\em J. Comp. Appl. Math.}, 6:19--26, 1980.

\bibitem{ELS19}
S.~Eidnes, L.~Li, and S.~Sato.
\newblock Linearly implicit structure-preserving schemes for {H}amiltonian
  systems.
\newblock {\em arXiv preprint arXiv:1901.03573}, 2019.

\bibitem{Furihata01}
D.~Furihata.
\newblock Finite-difference schemes for nonlinear wave equation that inherit
  energy conservation property.
\newblock {\em J. Comput. Appl. Math.}, 134:37--57, 2001.

\bibitem{FM2011}
D.~Furihata and T.~Matsuo.
\newblock {\em Discrete Variational Derivative Method: A Structure-Preserving
  Numerical Method for Partial Differential Equations}.
\newblock Chapman \& Hall/CRC, Boca Raton, 2011.

\bibitem{GCW14b}
Y.~Gong, J.~Cai, and Y.~Wang.
\newblock Some new structure-preserving algorithms for general multi-symplectic
  formulations of {H}amiltonian {PDE}s.
\newblock {\em J. Comput. Phys}, 279:80--102, 2014.

\bibitem{GZYW18}
Y.~Gong, J.~Zhao, X.~Yang, and Q.~Wang.
\newblock Fully discrete second-order linear schemes for hydrodynamic phase
  field models of binary viscous fluid flows with variable densities.
\newblock {\em SIAM J. Sci. Comput.}, 40:B138--B167, 2018.

\bibitem{Hairer00bit}
E.~Hairer.
\newblock Symmetric projection methods for differential equations on manifolds.
\newblock {\em BIT}, 40:726--734, 2000.

\bibitem{H10}
E.~Hairer.
\newblock Energy-preserving variant of collocation methods.
\newblock {\em J. Numer. Anal. Ind. Appl. Math.}, 5:73--84, 2010.

\bibitem{ELW06}
E.~Hairer, C.~Lubich, and G.~Wanner.
\newblock {\em Geometric Numerical Integration: Structure-Preserving Algorithms
  for Ordinary Differential Equations}.
\newblock Springer-Verlag, Berlin, 2nd edition, 2006.

\bibitem{HV2003}
W.~Hunddorfer and J.~G. Verwer.
\newblock {\em Numerical solution of Time-Dependent
  Advection-Diffusion-Reaction Equations}.
\newblock Springer Series in Computational Mathematics, vol. 33, Springer,
  Berlin, 2003.

\bibitem{JCW19jsc}
C.~Jiang, W.~Cai, and Y.~Wang.
\newblock A linearly implicit and local energy-preserving scheme for the
  sine-{G}ordon equation based on the invariant energy quadratization approach.
\newblock {\em J. Sci. Comput.}, 80:1629--1655, 2019.

\bibitem{JGCW19}
C.~Jiang, Y.~Gong, W.~Cai, and Y.~Wang.
\newblock A linearly implicit structure-preserving scheme for the
  {C}amassa-{H}olm equation based on multiple scalar auxiliary variables
  approach.
\newblock {\em J. Sci. Comput.}, 83:1--20, 2020.

\bibitem{JWG19}
C.~Jiang, Y.~Wang, and Y.~Gong.
\newblock Arbitrarily high-order energy-preserving schemes for the
  {C}amassa-{H}olm equation.
\newblock {\em Appl. Numer. Math.}, 151:85--97, 2020.

\bibitem{Kojima16bit}
H.~Kojima.
\newblock Invariants preserving schemes based on explicit {R}unge-{K}utta
  methods.
\newblock {\em BIT}, 56:1317--1337, 2016.

\bibitem{LWQ14}
H.~Li, Y.~Wang, and M.~Qin.
\newblock A sixth order averaged vector field method.
\newblock {\em J. Comput. Math.}, 34:479--498, 2016.

\bibitem{LQ95}
S.~Li and L.~Vu-Quoc.
\newblock Finite difference calculus invariant structure of a class of
  algorithms for the nonlinear {K}lein-{G}ordon equation.
\newblock {\em SIAM. J. Numer. Anal.}, 32:1839--1875, 1995.

\bibitem{LW16b}
Y.~Li and X.~Wu.
\newblock Exponential integrators preserving first integrals or {L}yapunov
  functions for conservative or dissipative systems.
\newblock {\em SIAM J. Sci. Comput.}, 38:A1876--A1895, 2016.

\bibitem{Matsuo03}
T.~Matsuo.
\newblock High-order schemes for conservative or dissipative systems.
\newblock {\em J. Comput. Appl. Math.}, 152:305--317, 2003.

\bibitem{MF01jcp}
T.~Matsuo and D.~Furihata.
\newblock Dissipative or conservative finite-difference schemes for
  complex-valued nonlinear partial differential equations.
\newblock {\em J. Comput. Phys.}, 171:425--447, 2001.

\bibitem{MQR99}
R.~I. McLachlan, G.~R.~W. Quispel, and N.~Robidoux.
\newblock Geometric integration using discrete gradients.
\newblock {\em Philos. Trans. R. Soc. A}, 357:1021--1045, 1999.

\bibitem{MB16}
Y.~Miyatake and J.~C. Butcher.
\newblock A characterization of energy-preserving methods and the construction
  of parallel integrators for {H}amiltonian systems.
\newblock {\em SIAM J. Numer. Anal.}, 54:1993--2013, 2016.

\bibitem{QM08}
G.~R.~W. Quispel and D.~I. McLaren.
\newblock A new class of energy-preserving numerical integration methods.
\newblock {\em J. Phys. A: Math. Theor.}, 41:045206, 2008.

\bibitem{Sanz-Sernabit88}
J.~M. Sanz-Serna.
\newblock Runge-{K}utta schemes for {H}amiltonian systems.
\newblock {\em BIT}, 28:877--883, 1988.

\bibitem{SCbook94}
J.~M. Sanz-Serna and M.~Calvo.
\newblock {\em Numerical Hamiltonian Problems}.
\newblock Chapman \& Hall, London, 1994.

\bibitem{ST06}
J.~Shen and T.~Tang.
\newblock {\em Spectral and High-Order Methods with Applications}.
\newblock Science Press, Beijing, 2006.

\bibitem{SXY18}
J.~Shen, J.~Xu, and J.~Yang.
\newblock The scalar auxiliary variable {(SAV)} approach for gradient.
\newblock {\em J. Comput. Phys.}, 353:407--416, 2018.

\bibitem{SXY19siamrev}
J.~Shen, J.~Xu, and J.~Yang.
\newblock A new class of efficient and robust energy stable schemes for
  gradient flows.
\newblock {\em SIAM Rev.}, 61:474--506, 2019.

\bibitem{SKV10}
Q.~Sheng, A.~Q.~M. Khaliq, and D.~A. Voss.
\newblock Numerical simulation of two-dimensional sine-{G}ordon solitons via a
  split cosine scheme.
\newblock {\em Math. Comput. Simulation}, 68:355--373, 2005.

\bibitem{TS12}
W.~Tang and Y.~Sun.
\newblock Time finite element methods: a unified framework for numerical
  discretizations of {ODE}s.
\newblock {\em Appl. Math. Comput.}, 219:2158--2179, 2012.

\bibitem{WWpla12}
B.~Wang and X.~Wu.
\newblock A new high precision energy-preserving integrator for system of
  oscillatory second-order differential equations.
\newblock {\em Phys. Lett. A}, 376:1185--1190, 2012.

\bibitem{WWQ08}
Y.~Wang, B.~Wang, and M.~Qin.
\newblock Local structure-preserving algorithms for partial differential
  equations.
\newblock {\em Sci. China Ser. A}, 51:2115--2136, 2008.

\bibitem{YZW17}
X.~Yang, J.~Zhao, and Q.~Wang.
\newblock Numerical approximations for the molecular beam epitaxial growth
  model based on the invariant energy quadratization method.
\newblock {\em J. Comput. Phys.}, 333:104--127, 2017.

\bibitem{ZhangQSaml19}
H.~Zhang, X.~Qian, and S.~Song.
\newblock Novel high-order energy-preserving diagonally implicit
  {R}unge-{K}utta schemes for nonlinear {H}amiltonian odes.
\newblock {\em Appl. Math. Lett.}, 102:106091, 2020.

\bibitem{ZhangQYS19}
H.~Zhang, X.~Qian, J.~Yan, and S.~Song.
\newblock Highly efficient invariant-conserving explicit {R}unge-{K}utta
  schemes for the nonlinear {H}amiltonian differential equations.
\newblock {\em researchgate.net}, 2019.

\bibitem{ZYGW17}
J.~Zhao, X.~Yang, Y.~Gong, and Q.~Wang.
\newblock A novel linear second order unconditionally energy stable scheme for
  a hydrodynamic-tensor model of liquid crystals.
\newblock {\em Comput. Methods Appl. Mech. Engrg.}, 318:803--825, 2017.

\end{thebibliography}

\end{document}